\documentclass{amsart}
\usepackage[dvips,final]{graphics}
\usepackage{array}
\usepackage{arydshln}
\usepackage[makeroom]{cancel}
 \usepackage[all]{xy}
 \usepackage{url}
\usepackage{multirow, blkarray}
\usepackage{booktabs}
\usepackage{textcomp}
 \usepackage[final]{epsfig}
 \usepackage{color}
\usepackage[T1]{fontenc}      
\usepackage[english,french]{babel}
\usepackage[utf8]{inputenc}
\usepackage{blindtext}

\usepackage{amsfonts,amscd,array, mathdots, epigraph}
\usepackage{amsmath}
\usepackage{amssymb}
\usepackage{amsthm}
\usepackage{mathrsfs}
\usepackage{stmaryrd}
\usepackage{slashbox}
\usepackage{diagbox}
\usepackage{enumitem}

\usepackage{ulem}
\usepackage{tikz}
\usepackage{xcolor}
\usepackage{multicol}
\definecolor{ufogreen}{rgb}{0.24, 0.82, 0.44}

\begin{document}


\newtheorem{theorem}{Théorème}[section]
\newtheorem{theore}{Théorème}
\newtheorem{definition}[theorem]{Définition}
\newtheorem{proposition}[theorem]{Proposition}
\newtheorem{corollary}[theorem]{Corollaire}
\newtheorem*{con}{Conjecture}
\newtheorem*{remark}{Remarque}
\newtheorem*{remarks}{Remarques}
\newtheorem*{pro}{Problème}
\newtheorem*{examples}{Exemples}
\newtheorem*{example}{Exemple}
\newtheorem{lemma}[theorem]{Lemme}


\title{$\lambda$-quiddités sur des produits directs d'anneaux}

\author{Flavien Mabilat}

\date{}


\address{
Flavien Mabilat,
Laboratoire de Mathématiques de Reims,
UMR9008 CNRS et Université de Reims Champagne-Ardenne, 
U.F.R. Sciences Exactes et Naturelles 
Moulin de la Housse - BP 1039 
51687 Reims cedex 2,
France
}
\email{flavien.mabilat@univ-reims.fr}

\maketitle

\selectlanguage{french}
\begin{abstract}

L'objectif de cet article est de poursuivre l'étude de la notion de $\lambda$-quiddité sur un anneau, apparue lors de l'étude des frises de Coxeter. Pour cela, on va s'intéresser ici aux situations où l'anneau utilisé peut être vu comme un produit direct d'anneaux commutatifs unitaires. En particulier, on va considérer les cas des produits directs contenant au moins deux anneaux de caractéristique $0$ et on va s'intéresser également à certains produits du type $\mathbb{Z}/n\mathbb{Z} \times \mathbb{Z}/m\mathbb{Z}$.

\end{abstract}

\selectlanguage{english}
\begin{abstract}

The aim of this article is to continue the study of the notion of $\lambda$-quiddity over a ring, which appeared during the study of Coxeter's friezes. For this, we will focus here on situations where the ring used can be seen as a direct product of unitary commutative rings. In particular, we will consider the cases of direct products of rings containing at least two rings of characteristic $0$ and we will also consider some products of the type $\mathbb{Z}/n\mathbb{Z} \times \mathbb{Z}/m\mathbb{Z}$.

\end{abstract}

\selectlanguage{french}

\thispagestyle{empty}

\noindent {\bf Mots clés:} groupe modulaire; $\lambda$-quiddité; irréductibilité; produit direct d'anneaux 
\\
\begin{flushright}
\og \textit{Penser, c'est être à la recherche d'un promontoire.} \fg
\\Michel de Montaigne, \textit{Essais.}
\end{flushright}

\section{Introduction}

Dans un certain nombre de problèmes mathématiques, on est rapidement amené à rencontrer les matrices suivantes : 
\[M_{n}(a_{1},\ldots,a_{n}):=\begin{pmatrix}
    a_{n} & -1 \\
    1  & 0 
   \end{pmatrix}\ldots \begin{pmatrix}
    a_{1} & -1 \\
    1  & 0 
   \end{pmatrix}.\]
\noindent En effet, ces dernières apparaissent naturellement dans l'étude du groupe modulaire, dans la formalisation matricielle des équations de Sturm-Liouville discrètes ou encore dans les formules concernant les réduites des fractions continues de Hirzebruch-Jung (voir par exemple l'introduction de \cite{O}). Parmi toutes ces applications, l'une des plus notable concerne la construction des frises de Coxeter. En effet, ces objets combinatoires, constitués d'arrangements de nombres dans le plan vérifiant certaines relations arithmétiques, ont de nombreuses ramifications et constituent aujourd'hui un vaste thème d'étude (voir par exemple \cite{Mo1}). Or, la construction de ces frises sur un sous-ensemble $E$ d'un anneau commutatif unitaire est liée aux solutions de l'équation suivante sur $E$ (voir par exemple \cite{CH} proposition 2.4) :
\[M_{n}(a_{1},\ldots,a_{n})=-Id.\]
\noindent Cela conduit naturellement à l'étude sur un anneau commutatif unitaire $A$ de l'équation généralisée ci-dessous (parfois appelée équation de Conway-Coxeter) :
\begin{equation}
\label{a}
\tag{$E_{A}$}
M_{n}(a_1,\ldots,a_n)=\pm Id.
\end{equation} 

\medskip

Les solutions de cette équation sont appelées $\lambda$-quiddités. Pour pouvoir les étudier, on utilise une notion de solutions irréductibles basée sur une opération sur les $n$-uplets d'éléments de l'ensemble considéré (voir \cite{C} et la section suivante). Une fois ce concept d'irréductibilité fixé, l'étude de \eqref{a} sur un sous-ensemble de $A$ stable par $+$ tend à se réduire à la bonne connaissance des $\lambda$-quiddités irréductibles sur l'ensemble considéré. Sur ce sujet, on dispose notamment d'une classification complète de ces dernières sur $\mathbb{Z}$ (voir \cite{C} Théorème 3.2), sur $\mathbb{Z}[\alpha]$ (voir \cite{M3} Théorème 2.7), sur $\mathbb{Z}/N\mathbb{Z}$ pour $2 \leq N \leq 6$ (voir \cite{M2} Théorème 2.5 et la section \ref{appli}), ainsi que sur un certain nombre de sous-groupes monogènes de $(\mathbb{C},+)$ (voir \cite{M4}). On a également une construction récursive des solutions de \eqref{a} lorsque les $a_{i}$ sont pris dans $\mathbb{N}^{*}$ (voir \cite{O} Théorèmes 1 et 2). Dans les cas où la classification n'est pas connue et semble hors de portée, on peut également chercher, de façon plus modeste, à connaître des propriétés vérifiées par les $\lambda$-quiddités, comme la présence d'éléments particuliers, le nombre de solutions de \eqref{a} de taille fixée ou encore les tailles possibles des solutions irréductibles.
\\
\\ \indent Afin d'avoir une meilleure connaissance des solutions de \eqref{a}, M.\ Cuntz a posé le problème de classifier les $\lambda$-quiddités irréductibles sur \og certains des sous-ensembles les plus intéressants de $\mathbb{C}$ \fg (voir \cite{C} problème 4.1). Ici, on propose d'élargir ce problème en considérant d'autres types d'ensembles : les produits directs d'anneaux commutatifs unitaires. En effet, en plus de l'intérêt intrinsèque des éléments qui seront développés, il existe de nombreux anneaux qui ne sont pas des produits directs mais qui sont isomorphes à des produits directs d'anneaux et à qui on pourra donc également appliquer les résultats obtenus. Dans ce qui suit, on s'intéressera notamment aux produits d'anneaux faisant intervenir au moins deux anneaux de caractéristique $0$ et à certains produits du type $\mathbb{Z}/n\mathbb{Z} \times \mathbb{Z}/m\mathbb{Z}$. Pour pouvoir faire cela, on commencera, dans la section \ref{RP}, par définir formellement tous les concepts évoqués ci-dessus et par donner les énoncés des résultats principaux de ce texte. Ensuite, dans la section \ref{preuve}, on effectuera les démonstrations de certains de ces théorèmes et on vérifiera que les isomorphismes d'anneaux unitaires conservent les propriétés des $\lambda$-quiddités. Pour terminer, on considérera dans la section \ref{appli} un certain nombre de cas particuliers, notamment ceux évoqués ci-dessus.

\section{Définitions et résultats principaux}
\label{RP}    

Avant d'entrer dans les détails des éléments énoncés ci-dessus, on va présenter dans cette partie les définitions essentielles à l'étude que l'on souhaite mener ainsi que les résultats principaux qui découleront de cette dernière. Dans tout ce qui suit, $(A,+, \times)$ est un anneau commutatif unitaire. Le neutre pour $+$ est $0_{A}$. Quant au neutre pour $\times$, on le note $1_{A}$. Si $n \in \mathbb{N}^{*}$, on note $n_{A}:=\sum_{i=1}^{n} 1_{A}$. Si $I$ est un ensemble et $(A_{i})_{i \in I}$ une famille d'anneaux commutatifs unitaires indexées par $I$, on note $\prod_{i \in I} A_{i}$ le produit direct des $A_{i}$. On munit cet ensemble des deux lois usuelles ci-dessous qui en font un anneau commutatif unitaire:
\begin{itemize}
\item $(a_{i})_{i \in I} + (b_{i})_{i \in I}=(a_{i}+b_{i})_{i \in I}$;
\item $(a_{i})_{i \in I} \times (b_{i})_{i \in I}=(a_{i}b_{i})_{i \in I}$.
\end{itemize}

\noindent Sauf mention contraire, $N$ désigne un entier naturel supérieur à 2. Si $a \in \mathbb{Z}$, on note, s'il n'y pas d'ambiguïté sur $N$, $\overline{a}:=a+N\mathbb{Z}$ la classe de $a$ modulo $N$. L'équation \eqref{a} devient alors :

\begin{equation}
\label{p}
\tag{$E_{N}$}
M_{n}(\overline{a_1},\ldots,\overline{a_n})=\pm Id.
\end{equation} 

\medskip

\noindent On débute cette partie par la définition formelle du concept de $\lambda$-quiddité sur $A$.

\begin{definition}[\cite{C}, définition 2.2]
\label{21}
Soit $n \in \mathbb{N}^{*}$. On dit que le $n$-uplet $(a_{1},\ldots,a_{n})$ d'éléments de $A$ est une $\lambda$-quiddité sur $A$ de taille $n$ si $(a_{1},\ldots,a_{n})$ est une solution de \eqref{a}, c'est-à-dire si $M_{n}(a_{1},\ldots,a_{n})=\pm Id.$ S'il n'y a pas d'ambiguïté on parlera simplement de $\lambda$-quiddité.

\end{definition}

Dans toute la suite, on parlera indistinctement de solution de \eqref{a} ou de $\lambda$-quiddité sur $A$. Par ailleurs, les solutions de \eqref{a} étant invariantes par permutations circulaires, on considérera, si $(a_{1},\ldots,a_{n})$ est une solution de taille $n$ de \eqref{a}, $a_{n+k}=a_{k}$.

\begin{definition}[\cite{C}, lemme 2.7]
\label{22}

Soient $(a_{1},\ldots,a_{n}) \in A^{n}$ et $(b_{1},\ldots,b_{m}) \in A^{m}$. On définit l'opération ci-dessous: \[(a_{1},\ldots,a_{n}) \oplus (b_{1},\ldots,b_{m})= (a_{1}+b_{m},a_{2},\ldots,a_{n-1},a_{n}+b_{1},b_{2},\ldots,b_{m-1}).\] Le $(n+m-2)$-uplet obtenu est appelé la somme de $(a_{1},\ldots,a_{n})$ avec $(b_{1},\ldots,b_{m})$.

\end{definition}

\begin{examples}

{\rm On donne ci-dessous quelques exemples de sommes dans $\mathbb{Z}$ :
\begin{itemize}
\item $(4,1,-1) \oplus (2,0,2,-3)=(1,1,1,0,2)$;
\item $(3,2,0,-1) \oplus (5,1,7,0)=(3,2,0,4,1,7)$;
\item $(3,4,0,2) \oplus (1,0,0,3,2)=(5,4,0,3,0,0,3)$;
\item $n \geq 2$, $(a_{1},\ldots,a_{n}) \oplus (0,0) = (0,0) \oplus (a_{1},\ldots,a_{n})=(a_{1},\ldots,a_{n})$.
\end{itemize}
}
\end{examples}

L'opération $\oplus$ possède la propriété très utile suivante : si $(b_{1},\ldots,b_{m})$ est une solution de \eqref{a} alors la somme $(a_{1},\ldots,a_{n}) \oplus (b_{1},\ldots,b_{m})$ est une solution de \eqref{a} si et seulement si $(a_{1},\ldots,a_{n})$ est une solution de \eqref{a} (voir \cite{C,WZ} et \cite{M2} proposition 3.7). De plus, si $M_{n}(a_{1},\ldots,a_{n})=\alpha Id$ et $M_{m}(b_{1},\ldots,b_{m})=\beta Id$ avec $\alpha, \beta \in \{\pm 1_{A}\}^{2}$, on a $M_{n+m-2}((a_{1},\ldots,a_{n}) \oplus (b_{1},\ldots,b_{m}))=-\alpha \beta Id$. En revanche, $\oplus$ n'est ni commutative ni associative (voir \cite{WZ} exemple 2.1) et les uplets d'éléments de $A$ différents de $(0,0)$ n'ont pas d'inverse pour $\oplus$.

\begin{definition}[\cite{C}, définition 2.5]
\label{23}

Soient $(a_{1},\ldots,a_{n}) \in A^{n}$ et $(b_{1},\ldots,b_{n}) \in A^{n}$. On note $(a_{1},\ldots,a_{n}) \sim (b_{1},\ldots,b_{n})$ si $(b_{1},\ldots,b_{n})$ est obtenu par permutations circulaires de $(a_{1},\ldots,a_{n})$ ou de $(a_{n},\ldots,a_{1})$.

\end{definition}

Une vérification directe permet de s'assurer que les deux propriétés suivantes sont vérifiées. D'une part, $\sim$ forme une relation d'équivalence sur les $n$-uplets d'éléments de $A$ (voir \cite{WZ} lemme 1.7). D'autre part, si $(a_{1},\ldots,a_{n}) \sim (b_{1},\ldots,b_{n})$ alors $(a_{1},\ldots,a_{n})$ est solution de \eqref{a} si et seulement si $(b_{1},\ldots,b_{n})$ l'est aussi (voir \cite{C} proposition 2.6). Plus précisément, si $M_{n}(a_{1},\ldots,a_{n})=\epsilon Id$, avec $\epsilon \in \{\pm 1_{A}\}$, et si $(a_{1},\ldots,a_{n}) \sim (b_{1},\ldots,b_{n})$ alors $M_{n}(b_{1},\ldots,b_{n})=\epsilon Id$. Par ailleurs, on peut préciser la définition ci-dessus. Soit $n \in \mathbb{N}^{*}$, on note $D_{n}$ le groupe diédral de cardinal $2n$. Celui-ci est engendré par deux éléments $s$ et $r$ avec $s$ d'ordre 2 et $r$ d'ordre $n$. $D_{n}$ agit sur $\{1,\ldots,n\}$ de la façon suivante. Soient $k \in \{0,\ldots,n-1\}$ et $i \in \{1,\ldots,n\}$. Il existe un unique $j \in \{1,\ldots,n\}$ tel que $(k+i)+n\mathbb{Z}=j+n\mathbb{Z}$. On pose $r^{k}.i=j$ et $s.i=n-i+1$. À la lueur de ces éléments, on pose, pour $\sigma \in D_{n}$, $(a_{1},\ldots,a_{n})^{\sigma}=(a_{\sigma .1},\ldots,a_{\sigma .n})$ et on a $(a_{1},\ldots,a_{n}) \sim (b_{1},\ldots,b_{n})$ si et seulement s'il existe $\sigma \in D_{n}$ tel que $(a_{1},\ldots,a_{n})^{\sigma}=(b_{1},\ldots,b_{n})$.
\\
\\ \noindent Les deux concepts précédents étant fixés, on peut maintenant définir la notion d'irréductibilité annoncée.

\begin{definition}[\cite{C}, définition 2.9]
\label{24}

Une solution $(c_{1},\ldots,c_{n})$ avec $n \geq 3$ de \eqref{a} est dite réductible s'il existe une solution de \eqref{a} $(b_{1},\ldots,b_{l})$ et un $m$-uplet $(a_{1},\ldots,a_{m})$ d'éléments de $A$ tels que \begin{itemize}
\item $(c_{1},\ldots,c_{n}) \sim (a_{1},\ldots,a_{m}) \oplus (b_{1},\ldots,b_{l})$;
\item $m \geq 3$ et $l \geq 3$.
\end{itemize}
Une solution est dite irréductible si elle n'est pas réductible.

\end{definition}

\begin{remark} 

{\rm $(0_{A},0_{A})$ est systématiquement considérée comme une solution réductible de \eqref{a}.}

\end{remark}

\indent Ainsi, pour chaque anneau, l'étude des solutions de \eqref{a} peut-être ramenée à la bonne connaissance des $\lambda$-quiddités irréductibles. On dispose déjà de nombreuses informations sur ces dernières sur un certain nombre d'anneaux (voir \cite{C, M2, M3}). Ici, on souhaite s'intéresser au cas où l'anneau considéré est un produit direct d'anneaux. Pour cela, on va notamment démontrer le résultat suivant : 

\begin{theorem}
\label{25}

Soient $I$ un ensemble et $(A_{i})_{i \in I}$ une famille d'anneaux commutatifs unitaires indexées par $I$. Soient $n \in \mathbb{N}^{*}$, $B=\prod_{i \in I} A_{i}$ et $((a_{i,1})_{i},\ldots,(a_{i,n})_{i}) \in B^{n}$. 
\\
\\i) $((a_{i,1})_{i},\ldots,(a_{i,n})_{i})$ est une $\lambda$-quiddité si et seulement s'il existe $(\epsilon_{A_{i}})_{i} \in \{-1_{B},1_{B}\}$ tel que pour tout $i \in I$ $M_{n}(a_{i,1},\ldots,a_{i,n})=\epsilon_{A_{i}} Id$.
\\
\\ii) $((a_{i,1})_{i},\ldots,(a_{i,n})_{i})$ est une $\lambda$-quiddité réductible si et seulement si les conditions suivantes sont vérifiées :
\begin{itemize}
\item il existe $3 \leq l \leq n-1$ tel que pour tout $i \in I$ il existe une $\lambda$-quiddité $(b_{i,1},\ldots,b_{i,l})$ sur $A_{i}$;
\item il existe $(\epsilon_{A_{i}})_{i} \in \{-1_{B},1_{B}\}$ tel que pour tout $i \in I$ $M_{l}(b_{i,1},\ldots,b_{i,l})=\epsilon_{A_{i}} Id$;
\item il existe $\sigma \in D_{n}$ tel que pour tout $i \in I$ $(a_{i,1},\ldots,a_{i,n})^{\sigma}=(c_{i,1},\ldots,c_{i,n+2-l}) \oplus (b_{i,1},\ldots,b_{i,l})$ avec $(c_{i,1},\ldots,c_{i,n+2-l}) \in A_{i}^{n+2-l}$.
\end{itemize}

\noindent Si une $\lambda$-quiddité $((a_{i,1})_{i},\ldots,(a_{i,n})_{i})$ sur $\prod_{i \in I} A_{i}$ vérifie les trois conditions précédentes on dira que les $\lambda$-quiddités $(a_{i,1},\ldots,a_{i,n})$ sont simultanément réductibles.

\end{theorem}

\begin{remarks}

{\rm 
i) Si tous les anneaux $A_{i}$, sauf éventuellement un, sont de caractéristique 2, la condition portant sur l'existence d'un $\epsilon$ commun est automatiquement vérifiée.
\\
\\ii) Si $\alpha$ est une bijection de $I$ alors $((a_{i,1})_{i},\ldots,(a_{i,n})_{i})$ est une $\lambda$-quiddité irréductible sur $\prod_{i \in I} A_{i}$ si et seulement si $((a_{\alpha(i),1})_{i},\ldots,(a_{\alpha(i),n})_{i})$ est une $\lambda$-quiddité irréductible sur $\prod_{i \in I} A_{\alpha(i)}$.
}

\end{remarks}

Ce théorème sera démontré dans la section suivante. En particulier, on pourra appliquer ce dernier dans des cas où l'anneau considéré n'est pas un produit direct mais est isomorphe à un produit direct. La vérification formelle du bon comportement des $\lambda$-quiddités par rapport aux isomorphismes d'anneaux unitaires sera également faite dans la section \ref{preuve}. Par ailleurs, on notera par la même occasion que l'existence d'un isomorphisme de groupes entre deux anneaux commutatifs unitaires n'est pas suffisante. Pour illustrer cela, on considérera les deux résultats ci-dessous :

\begin{theorem}
\label{26}

Les $\lambda$-quiddités irréductibles sur $(\mathbb{Z}/2\mathbb{Z}) \times (\mathbb{Z}/2\mathbb{Z})$ sont (à permutations cycliques près) :
\begin{itemize}
\item $((\overline{1},\overline{1}),(\overline{1},\overline{1}),(\overline{1},\overline{1}))$;
\item $((\overline{0},\overline{0}),(\overline{0},\overline{0}),(\overline{0},\overline{0}),(\overline{0},\overline{0})), ((\overline{0},\overline{0}),(\overline{0},\overline{1}),(\overline{0},\overline{0}),(\overline{0},\overline{1})), ((\overline{0},\overline{0}),(\overline{1},\overline{0}),(\overline{0},\overline{0}),(\overline{1},\overline{0})), 
\\((\overline{1},\overline{0}),(\overline{0},\overline{1}),(\overline{1},\overline{0}),(\overline{0},\overline{1}))$;
\item $((\overline{1},\overline{0}),(\overline{1},\overline{0}),(\overline{1},\overline{0}),(\overline{1},\overline{0}),(\overline{1},\overline{0}),(\overline{1},\overline{0})), ((\overline{0},\overline{1}),(\overline{0},\overline{1}),(\overline{0},\overline{1}),(\overline{0},\overline{1}),(\overline{0},\overline{1}),(\overline{0},\overline{1}))$.
\end{itemize}

\end{theorem}

\begin{theorem}
\label{27}

On considère le corps à 4 éléments :
\[\mathbb{F}_{4}=\frac{\mathbb{Z}/2\mathbb{Z}[X]}{<X^{2}+\overline{1}>}=\{\overline{0},\overline{1},X,X+\overline{1}\}.\]
\noindent Les $\lambda$-quiddités irréductibles sur $\mathbb{F}_{4}$ sont (à permutations cycliques près) :
\begin{itemize}
\item $(\overline{1},\overline{1},\overline{1})$;
\item $(\overline{0},\overline{0},\overline{0},\overline{0})$, $(\overline{0},X,\overline{0},X)$, $(\overline{0},X+\overline{1},\overline{0},X+\overline{1})$;
\item $(X,X,X,X,X)$, $(X+\overline{1},X+\overline{1},X+\overline{1},X+\overline{1},X+\overline{1})$;
\item $(X,X+\overline{1},X,X+\overline{1},X,X+\overline{1})$;
\item $(X,X,X+\overline{1},X+\overline{1},X,X,X+\overline{1},X+\overline{1})$;
\item $(X,X,X+\overline{1},X,X,X+\overline{1},X,X,X+\overline{1})$, $(X+\overline{1},X+\overline{1},X,X+\overline{1},X+\overline{1},X,X+\overline{1},X+\overline{1},X)$.
\end{itemize}

\end{theorem}

\begin{remark}
{\rm On dispose également de formules permettant de connaître le nombre de $\lambda$-quiddités de taille fixée sur $\mathbb{F}_{4}$ (voir \cite{CM}).
}
\end{remark}

Ensuite, on utilisera le théorème \ref{25} pour étudier concrètement plusieurs anneaux produits. En particulier, on effectuera la classification des $\lambda$-quiddités irréductibles pour certains produits d'anneaux du type $\mathbb{Z}/N\mathbb{Z}$. On démontrera également le résultat suivant :

\begin{theorem}
\label{28}

Soient $I$ un ensemble contenant au moins deux éléments et $(A_{i})$ une famille d'anneaux commutatifs unitaires. On suppose qu'au moins deux des anneaux $A_{i}$ sont de caractéristique 0. Soit $n \geq 3$. Il existe une $\lambda$-quiddité irréductible de taille $n$ sur $\prod_{i \in I} A_{i}$.

\end{theorem}

\begin{corollary}
\label{29}

Soit $n \geq 3$. Il existe une $\lambda$-quiddité irréductible de taille $n$ sur $\mathbb{Z} \times \mathbb{Z}$.

\end{corollary}

\noindent Ces deux résultats seront prouvés dans la section \ref{appli}.

\section{Démonstration du théorème \ref{25} et utilisation des isomorphismes d'anneaux unitaires}
\label{preuve}

L'objectif de cette section est d'effectuer la preuve du théorème \ref{25} et de regarder en détail la façon dont on peut utiliser les isomorphismes d'anneaux unitaires. Avant de faire cela, on introduit les notations suivantes. On pose $K_{-1}=0_{A}$, $K_{0}=1_{A}$ et on note pour $i \geq 1$ et $(a_{1},\ldots,a_{i}) \in A^{i}$ : \[K_{i}(a_{1},\ldots,a_{i}):=
\left|
\begin{array}{cccccc}
a_1&1_{A}&&&\\[4pt]
1_{A}&a_{2}&1_{A}&&\\[4pt]
&\ddots&\ddots&\!\!\ddots&\\[4pt]
&&1_{A}&a_{i-1}&\!\!\!\!\!1_{A}\\[4pt]
&&&\!\!\!\!\!1_{A}&\!\!\!\!a_{i}
\end{array}
\right|.\] $K_{i}(a_{1},\ldots,a_{i})$ est le continuant de $a_{1},\ldots,a_{i}$. On dispose de l'égalité suivante (voir par exemple \cite {O}): \[M_{n}(a_{1},\ldots,a_{n})=\begin{pmatrix}
    K_{n}(a_{1},\ldots,a_{n}) & -K_{n-1}(a_{2},\ldots,a_{n}) \\
    K_{n-1}(a_{1},\ldots,a_{n-1})  & -K_{n-2}(a_{2},\ldots,a_{n-1}) 
   \end{pmatrix}.\]

\subsection{Démonstration du théorème \ref{25}}

Soient $I$ un ensemble et $(A_{i})_{i \in I}$ une famille d'anneaux commutatifs unitaires indexées par $I$. Soient $n \in \mathbb{N}^{*}$ et $((a_{i,1})_{i},\ldots,(a_{i,n})_{i}) \in (\prod_{i \in I} A_{i})^{n}$. On note $B=\prod_{i \in I} A_{i}$.
\\
\\i) Par définition des opérations $+$ et $\times$ sur $\prod_{i \in I} A_{i}$, on a :
\[M_{n}((a_{i,1})_{i},\ldots,(a_{i,n})_{i})=\begin{pmatrix}
    (K_{n}(a_{i,1},\ldots,a_{i,n}))_{i} & (-K_{n-1}(a_{i,2},\ldots,a_{i,n}))_{i} \\
    (K_{n-1}(a_{i,1},\ldots,a_{i,n-1}))_{i}  & (-K_{n-2}(a_{i,2},\ldots,a_{i,n-1}))_{i} 
   \end{pmatrix}.\] 
	
\noindent Si $((a_{i,1})_{i},\ldots,(a_{i,n})_{i})$ est une $\lambda$-quiddité. Il existe $\epsilon=(\epsilon_{A_{i}})_{i} \in \{\pm 1_{B}\}$ tel que 
\[(K_{n}(a_{i,1},\ldots,a_{i,n}))_{i}=(-K_{n-2}(a_{i,2},\ldots,a_{i,n-1}))_{i}=\epsilon\]
\noindent et 
\[(K_{n-1}(a_{i,1},\ldots,a_{i,n-1}))_{i}=(-K_{n-1}(a_{i,2},\ldots,a_{i,n}))_{i}=0_{B}.\]
\noindent Donc, pour tout $i \in I$, on dispose des égalités : $K_{n}(a_{i,1},\ldots,a_{i,n})=-K_{n-2}(a_{i,2},\ldots,a_{i,n-1})=\epsilon_{A_{i}}$ et $K_{n-1}(a_{i,1},\ldots,a_{i,n-1})=-K_{n-1}(a_{i,2},\ldots,a_{i,n})=0_{A_{i}}$. Ainsi, pour tout $i \in I$, $M_{n}(a_{i,1},\ldots,a_{i,n})=\epsilon_{A_{i}} Id$. 
\\
\\S'il existe $\epsilon=(\epsilon_{A_{i}})_{i} \in \{\pm 1_{B}\}$ tel que pour tout $i \in I$, $M_{n}(a_{i,1},\ldots,a_{i,n})=\epsilon_{A_{i}} Id$. Pour tout $i \in I$, on a $K_{n}(a_{i,1},\ldots,a_{i,n})=-K_{n-2}(a_{i,2},\ldots,a_{i,n-1})=\epsilon_{A_{i}}$, $K_{n-1}(a_{i,1},\ldots,a_{i,n-1})=K_{n-1}(a_{i,2},\ldots,a_{i,n})=0_{A_{i}}$. Donc, puisque ces égalités sont vraies pour tout $i$, $(K_{n}(a_{i,1},\ldots,a_{i,n}))_{i}=(-K_{n-2}(a_{i,2},\ldots,a_{i,n-1}))_{i}=\epsilon$ et $(K_{n-1}(a_{i,1},\ldots,a_{i,n-1}))_{i}=(K_{n-1}(a_{i,2},\ldots,a_{i,n}))_{i}=0_{B}$ et $M_{n}((a_{i,1})_{i},\ldots,(a_{i,n})_{i})=\epsilon Id$.
\\
\\ii) Si $((a_{i,1})_{i},\ldots,(a_{i,n})_{i})$ est une $\lambda$-quiddité réductible. Il existe $3 \leq l \leq n-1$, $\sigma \in D_{n}$, $((b_{i,1})_{i},\ldots,(b_{i,l})_{i})$ une $\lambda$-quiddité sur $(\prod_{i \in I} A_{i})$ et $((c_{i,1})_{i},\ldots,(c_{i,n+2-l})_{i}) \in (\prod_{i \in I} A_{i})^{n+2-l}$ tels que 
\[((a_{i,1})_{i},\ldots,(a_{i,n})_{i})^{\sigma}=((c_{i,1})_{i},\ldots,(c_{i,n+2-l})_{i}) \oplus ((b_{i,1})_{i},\ldots,(b_{i,l})_{i}).\] 

\noindent La relation ci-dessus donne pour tout $i \in I$ $(a_{i,1},\ldots,a_{i,n})^{\sigma}=(c_{i,1},\ldots,c_{i,n+2-l}) \oplus (b_{i,1},\ldots,b_{i,l})$. De plus, par i), il existe $(\epsilon_{A_{i}})_{i} \in \{\pm 1_{B}\}$ tel que pour tout $i \in I$ $M_{l}(b_{i,1},\ldots,b_{i,l})=\epsilon_{A_{i}} Id$. 
\\
\\On suppose maintenant que les condition suivantes sont vérifiées :
\begin{itemize}
\item il existe $3 \leq l \leq n-1$ tel que pour tout $i \in I$ il existe une $\lambda$-quiddité $(b_{i,1},\ldots,b_{i,l})$ sur $A_{i}$;
\item il existe $\epsilon=(\epsilon_{A_{i}})_{i} \in \{\pm 1_{B}\}$ tel que pour tout $i \in I$ $M_{l}(b_{i,1},\ldots,b_{i,l})=\epsilon_{A_{i}} Id$;
\item il existe $\sigma \in D_{n}$ tel que pour tout $i \in I$ $(a_{i,1},\ldots,a_{i,n})^{\sigma}=(c_{i,1},\ldots,c_{i,n+2-l}) \oplus (b_{i,1},\ldots,b_{i,l})$ avec $(c_{i,1},\ldots,c_{i,n+2-l}) \in A_{i}^{n+2-l}$.
\end{itemize}

\noindent La relation ci-dessus donne $((a_{i,1})_{i},\ldots,(a_{i,n})_{i})^{\sigma}=((c_{i,1})_{i},\ldots,(c_{i,n+2-l})_{i}) \oplus ((b_{i,1})_{i},\ldots,(b_{i,l})_{i})$. De plus, par i), $((b_{i,1})_{i},\ldots,(b_{i,l})_{i})$ est une $\lambda$-quiddité sur $(\prod_{i \in I} A_{i})$ de taille $3 \leq l \leq n-1$. Donc, $((a_{i,1})_{i},\ldots,(a_{i,n})_{i})$ est réductible.

\qed

\subsection{Morphismes entre anneaux unitaires}

L'ingrédient principal de cette sous-partie est la proposition ci-dessous :

\begin{proposition}
\label{31}

Soient $A$ et $B$ deux anneaux commutatifs unitaires et $f:A\longrightarrow B$ un morphisme d'anneaux unitaires.
\\
\\i) Soient $n \in \mathbb{N}^{*}$ et $(a_{1},\ldots,a_{n}) \in A^{n}$ une $\lambda$-quiddité sur $A$. $(f(a_{1}),\ldots,f(a_{n})) \in B^{n}$ est une $\lambda$-quiddité sur $B$.
\\
\\ii) De plus, si $(f(a_{1}),\ldots,f(a_{n}))$ est une $\lambda$-quiddité irréductible sur $B$ alors $(a_{1},\ldots,a_{n})$ est une $\lambda$-quiddité irréductible sur $A$

\end{proposition}

\begin{proof}

i) Soit $(a_{1},\ldots,a_{n}) \in A^{n}$ une $\lambda$-quiddité sur l'anneau $A$. Il existe $\epsilon \in \{1_{A}, -1_{A}\}$ tel que $M_{n}(a_{1},\ldots,a_{n})=\epsilon Id$. De plus, $K_{n}(a_{1},\ldots,a_{n})$ est un polynôme à  $n$ variables à coefficients entiers évalué en $(a_{1},\ldots,a_{n})$. Comme $f$ est un morphisme d'anneaux unitaires, on a les égalités suivantes : $K_{n}(f(a_{1}),\ldots,f(a_{n}))=f(K_{n}(a_{1},\ldots,a_{n}))$, $f(0_{A})=0_{B}$ et $f(\epsilon)=\omega \in \{-1_{B}, 1_{B}\}$. Ainsi,
\begin{eqnarray*}
M_{n}(f(a_{1}),\ldots,f(a_{n})) &=& \begin{pmatrix}
    K_{n}(f(a_{1}),\ldots,f(a_{n})) & -K_{n-1}(f(a_{2}),\ldots,f(a_{n})) \\
    K_{n-1}(f(a_{1}),\ldots,f(a_{n-1}))  & -K_{n-2}(f(a_{2}),\ldots,f(a_{n-1})) 
   \end{pmatrix} \\
	                              &=& \begin{pmatrix}
    f(K_{n}(a_{1},\ldots,a_{n})) & f(-K_{n-1}(a_{2},\ldots,a_{n})) \\
    f(K_{n-1}(a_{1},\ldots,a_{n-1}))  & f(-K_{n-2}(a_{2},\ldots,a_{n-1})) 
   \end{pmatrix} \\
	                              &=& \begin{pmatrix}
    f(\epsilon) & f(0_{A}) \\
    f(0_{A})  & f(\epsilon) 
   \end{pmatrix} \\
	                              &=& \omega Id.\\
\end{eqnarray*}

\noindent ii) On raisonne par contraposée. Si $(a_{1},\ldots,a_{n})$ est une $\lambda$-quiddité réductible sur $A$. Il existe $l, l' \geq 3$ et deux $\lambda$-quiddités $(b_{1},\ldots,b_{l}) \in A^{l}$ et $(c_{1},\ldots,c_{l'}) \in A^{l'}$ tels que :
\[(a_{1},\ldots,a_{n}) \sim (b_{1},\ldots,b_{l}) \oplus (c_{1},\ldots,c_{l'})=(b_{1}+c_{l'},b_{2},\ldots,b_{l-1},b_{l}+c_{1},c_{2},\ldots,c_{l'-1}).\]

\noindent On a :
\begin{eqnarray*}
(f(a_{1}),\ldots,f(a_{n})) &\sim& (f(b_{1}+c_{l'}),f(b_{2}),\ldots,f(b_{l-1}),f(b_{l}+c_{1}),f(c_{2}),\ldots,f(c_{l'-1})) \\
                           &=& (f(b_{1})+f(c_{l'}),f(b_{2}),\ldots,f(b_{l-1}),f(b_{l})+f(c_{1}),f(c_{2}),\ldots,f(c_{l'-1})) \\
										       &=& (f(b_{1}),\ldots,f(b_{l})) \oplus (f(c_{1}),\ldots,f(c_{l'})).
\end{eqnarray*}

\noindent Or, par i), $(f(c_{1}),\ldots,f(c_{l'}))$ est une $\lambda$-quiddité sur $B$. Ainsi, $(f(a_{1}),\ldots,f(a_{n}))$ est une $\lambda$-quiddité réductible sur $B$.

\end{proof}

\noindent Grâce à ce résultat, on peut obtenir facilement les deux résultats suivants :

\begin{corollary}
\label{32}

Soit $A$ un anneau commutatif unitaire de caractéristique $p$, avec $p$ premier. Si $(a_{1},\ldots,a_{n})$ est une $\lambda$-quiddité sur $A$ alors $(a_{1}^{p},\ldots,a_{n}^{p})$ est une $\lambda$-quiddité sur $A$.

\end{corollary}

\begin{proof}

Il suffit d'appliquer le résultat précédent au morphisme de Frobenius.

\end{proof}

\begin{corollary}
\label{33}

Soient $A$ et $B$ deux anneaux commutatifs unitaires. On suppose qu'il existe un isomorphisme d'anneaux unitaires $f:A\longrightarrow B$. Soient $n \in \mathbb{N}^{*}$ et $(a_{1},\ldots,a_{n}) \in A^{n}$.
\\
\\i) $(a_{1},\ldots,a_{n}) \in A^{n}$ est une $\lambda$-quiddité sur $A$ si et seulement si $(f(a_{1}),\ldots,f(a_{n})) \in B^{n}$ est une $\lambda$-quiddité sur $B$.
\\
\\ii) De plus, $(a_{1},\ldots,a_{n})$ est une $\lambda$-quiddité irréductible sur $A$ si et seulement si $(f(a_{1}),\ldots,f(a_{n}))$ est une $\lambda$-quiddité irréductible sur $B$.

\end{corollary}

\begin{proof}

La partie i) découle de la proposition \ref{31} appliquée à $f$ et $f^{-1}$. De même, on a que si $(f(a_{1}),\ldots,f(a_{n}))$ est une $\lambda$-quiddité irréductible sur $B$ alors $(a_{1},\ldots,a_{n})$ est une $\lambda$-quiddité irréductible sur $A$ et si $(a_{1},\ldots,a_{n})=(f^{-1}(f(a_{1})),\ldots,f^{-1}(f(a_{n})))$ est une $\lambda$-quiddité irréductible sur $A$ alors $(f(a_{1}),\ldots,f(a_{n}))$ est une $\lambda$-quiddité irréductible sur $B$.

\end{proof}

\begin{corollary}
\label{34}

Soit $\mathbb{F}$ un corps fini de caractéristique $p$, avec $p$ premier. $(a_{1},\ldots,a_{n})$ est une $\lambda$-quiddité irréductible sur $\mathbb{F}$ si et seulement si $(a_{1}^{p},\ldots,a_{n}^{p})$ l'est aussi.

\end{corollary}

\begin{proof}

Il suffit d'appliquer le résultat précédent au morphisme de Frobenius qui est un automorphisme dans le cas des corps finis.

\end{proof}

\begin{remark}
{\rm La proposition \ref{31} permet également de retrouver le résultat assez immédiat suivant : si $A$ et $B$ sont deux anneaux commutatifs unitaires vérifiant $A \subset B$ alors une $\lambda$-quiddité sur $A$ irréductible sur $B$ est irréductible sur $A$. Pour montrer cela avec la proposition \ref{31}, il suffit de considérer l'injection canonique.
}
\end{remark}

Le corollaire \ref{33} permet donc de transférer la classification des $\lambda$-quiddités irréductibles sur un anneau commutatif unitaire à d'autres anneaux via un isomorphisme d'anneaux unitaires. Toutefois, comme indiqué dans la section précédente, l'existence d'un isomorphisme de groupes n'est pas suffisante. Pour cela, on va démontrer les théorèmes \ref{26} et \ref{27}. 
\\
\\ \indent Avant de d'effectuer les preuves en détail, on a besoin de quelques informations sur les solutions de \eqref{a} pour les petites valeurs de $n$ (voir par exemple \cite{M2} section 3.1) :

\begin{lemma}
\label{35}

\begin{itemize}
\item \eqref{a} n'a pas de solution de taille 1.
\item $(0_{A},0_{A})$ est l'unique solution de \eqref{a} de taille 2.
\item $(1_{A},1_{A},1_{A})$ et $(-1_{A},-1_{A},-1_{A})$ sont les seules solutions de \eqref{a} de taille 3 et elles sont irréductibles.
\item Les solutions de \eqref{a} de taille 4 sont de la forme $(-a,b,a,-b)$ avec $ab=0$ et $(a,b,a,b)$ avec $ab=2_{A}$. 
\item Les solutions de \eqref{a} de taille supérieure à 4 contenant $\pm 1_{A}$ sont réductibles.
\item Une solution de \eqref{a} de taille 4 est irréductible si et seulement si elle ne contient pas $\pm 1_{A}$.
\item Les solutions de \eqref{a} de taille supérieure à 5 contenant $0_{A}$ sont réductibles.
\end{itemize}

\end{lemma}

\begin{proof}[Démonstration du théorème \ref{26}]

À la lueur du théorème \ref{25}, on voit qu'il suffit de trouver les couples de solutions de $(E_{2})$ qui ne sont pas simultanément réductibles. En utilisant le théorème \ref{25} et le lemme \ref{35}, on constate que les solutions de \eqref{a} irréductibles de taille 3 et 4 sont celles données dans l'énoncé.
\\
\\Soient $n \geq 5$, $(\overline{a_{1}},\ldots,\overline{a_{n}})$ et $(\overline{b_{1}},\ldots,\overline{b_{n}})$ deux $\lambda$-quiddités sur $\mathbb{Z}/2\mathbb{Z}$. On suppose qu'elles ne sont pas simultanément réductibles.
\\
\\On suppose pour commencer qu'il existe $1 \leq j \leq n$ tel que $\overline{a_{j}}=\overline{1}$. Si $\overline{b_{j}}=\overline{1}$ alors les deux solutions sont simultanément réductibles, puisqu'on a :
\[\left\{
    \begin{array}{ll}
        (\overline{a_{j+1}},\ldots,\overline{a_{n}},\overline{a_{1}},\ldots,\overline{a_{j}})=(\overline{a_{j+1}+1},\ldots,\overline{a_{n}},\overline{a_{1}},\ldots,\overline{a_{j-1}+1}) \oplus (\overline{1},\overline{1},\overline{1}); \\
        (\overline{b_{j+1}},\ldots,\overline{b_{n}},\overline{b_{1}},\ldots,\overline{b_{j}})=(\overline{b_{j+1}+1},\ldots,\overline{b_{n}},\overline{b_{1}},\ldots,\overline{b_{j-1}+1}) \oplus (\overline{1},\overline{1},\overline{1}).
    \end{array}
\right. \\ \]

\noindent Donc, $\overline{b_{j}}=\overline{0}$. Si $\overline{a_{j+1}}=\overline{0}$ les deux solutions sont simultanément réductibles. En effet, on a
\[\left\{
    \begin{array}{ll}
        (\overline{a_{j+2}},\ldots,\overline{a_{n}},\overline{a_{1}},\ldots,\overline{a_{j}},\overline{a_{j+1}})=(\overline{a_{j+2}+1},\ldots,\overline{a_{n}},\overline{a_{1}},\ldots,\overline{a_{j-1}}) \oplus (\overline{0},\overline{1},\overline{0},\overline{1}); \\
        (\overline{b_{j+2}},\ldots,\overline{b_{n}},\overline{b_{1}},\ldots,\overline{b_{j}},\overline{b_{j+1}})=(\overline{b_{j+2}},\ldots,\overline{b_{n}},\overline{b_{1}},\ldots,\overline{b_{j-1}+b_{j+1}}) \oplus (\overline{b_{j+1}},\overline{0},\overline{b_{j+1}},\overline{0}).
    \end{array}
\right. \\ \]

\noindent Ainsi, $\overline{a_{j+1}}=\overline{1}$. Si $\overline{b_{j+1}}=\overline{1}$ alors les deux solutions sont simultanément réductibles. Donc, $\overline{b_{j+1}}=\overline{0}$. En continuant ainsi, on voit que tous les $\overline{a_{j}}$ sont égaux à $\overline{1}$ et tous les $\overline{b_{j}}$ sont égaux à $\overline{0}$. Comme $M_{1}(\overline{1})$ est d'ordre 3 dans $PSL_{2}(\mathbb{Z}/2\mathbb{Z})$ et $M_{1}(\overline{0})$ est d'ordre 2 dans $PSL_{2}(\mathbb{Z}/2\mathbb{Z})$, on a $n$ divisible par 6. Si $n \geq 12$ alors les deux solutions sont simultanément réductibles. En effet, on a :
\[\left\{
    \begin{array}{ll}
        \underbrace{(\overline{1},\ldots,\overline{1})}_{n}=\underbrace{(\overline{0},\overline{1},\ldots,\overline{1},\overline{0})}_{n-4} \oplus (\overline{1},\overline{1},\overline{1},\overline{1},\overline{1},\overline{1}); \\
        \underbrace{(\overline{0},\ldots,\overline{0})}_{n}=\underbrace{(\overline{0},\ldots,\overline{0})}_{n-4} \oplus (\overline{0},\overline{0},\overline{0},\overline{0},\overline{0},\overline{0}).
    \end{array}
\right. \\ \]

\noindent En revanche, si $n=6$ les deux solutions ne sont pas simultanément réductibles. En effet, pour réduire simultanément les deux solutions, il nous faudrait une solution de taille 3 de la forme $(\overline{x},\overline{0},\overline{y})$ ou une solution de taille 4 de la forme $(\overline{x},\overline{1},\overline{1},\overline{y})$, ce qui n'existe pas.
\\
\\Si, pour tout $1 \leq j \leq n$, $\overline{a_{j}}=\overline{0}$. S'il existe $1 \leq k \leq n$ tel que $\overline{b_{k}}=\overline{0}$ alors les deux solutions sont simultanément réductibles. Donc, tous les $\overline{b_{j}}$ sont égaux à $\overline{1}$ et on aboutit, en procédant comme précédemment, à $(\overline{a_{1}},\ldots,\overline{a_{n}})=(\overline{0},\overline{0},\overline{0},\overline{0},\overline{0},\overline{0})$ et $(\overline{b_{1}},\ldots,\overline{b_{n}})=(\overline{1},\overline{1},\overline{1},\overline{1},\overline{1},\overline{1})$.

\end{proof}

On va maintenant effectuer la preuve du théorème \ref{27}. Notons que, puisque tous les corps à 4 éléments sont isomorphes, on peut, grâce au corollaire \ref{33}, utiliser $\mathbb{F}_{4}=\frac{\mathbb{Z}/2\mathbb{Z}[X]}{<X^{2}+\overline{1}>}$.

\begin{proof}[Démonstration du théorème \ref{27}]

Par le lemme \ref{35}, les solutions de taille 3 et 4 données dans l'énoncé sont les seules $\lambda$-quiddités irréductibles sur $\mathbb{F}_{4}$ de taille 3 et 4. Par le lemme \ref{35}, les solutions irréductibles de $(E_{\mathbb{F}_{4}})$ de taille supérieure à 5 ne contiennent que des $X$ et des $X+\overline{1}$. On vérifie que les solutions de taille comprise entre 5 et 9 ne contenant que $X$ et $X+\overline{1}$ sont celles données dans l'énoncé. Une solution réductible de taille 5 ou 6 contenant nécessairement $\overline{0}$ ou $\overline{1}$, $(X,X,X,X,X)$, $(X+\overline{1},X+\overline{1},X+\overline{1},X+\overline{1},X+\overline{1})$ et $(X,X+\overline{1},X,X+\overline{1},X,X+\overline{1})$ sont irréductibles. Si les solutions de l'énoncé de taille 8 étaient réductibles, elles pourraient être réduite par une solution de taille comprise entre 3 et 7, ce qui n'est pas le cas. De même, les les solutions de l'énoncé de taille 9 sont irréductibles.
\\
\\Considérons une $\lambda$-quiddité $(a_{1},\ldots,a_{n})$ sur $\mathbb{F}_{4}$ de taille supérieure à 10. Si celle-ci contient $\overline{0}$ ou $\overline{1}$, elle est réductible (lemme \ref{35}). On suppose donc qu'elle ne contient que $X$ et $X+\overline{1}$. On va montrer que celle-ci peut être réduite avec une solution de taille comprise entre 5 et 9, ce qui entraîne la réductibilité de la solution puisque $n \geq 10$. Afin de simplifier les notations, on note $Y:=X+\overline{1}$.
\\
\\Quitte à utiliser le corollaire \ref{34}, on peut supposer que $a_{2}=X$, puisque $\mathbb{F}_{4}$ est de caractéristique 2 et $(X+\overline{1})^{2}=X$. 
\\
\\On suppose pour le moment $a_{3}=X$. Si $a_{4}=X$ alors on peut réduire $(a_{1},\ldots,a_{n})$ avec $(X,X,X,X,X)$. Si $a_{4}=Y$, on s'intéresse aux termes suivants :
\begin{itemize}
\item si $a_{5}=a_{6}=Y$ alors on peut réduire $(a_{1},\ldots,a_{n})$ avec $(Y,Y,Y,Y,Y)$;
\item si $a_{5}=X$ et $a_{6}=Y$ alors on peut réduire $(a_{1},\ldots,a_{n})$ avec $(Y,X,Y,X,Y,X)$;
\item si $a_{5}=a_{6}=X$, on s'intéresse aux termes suivants :
\begin{itemize}[label=$\circ$]
\item si $a_{7}=X$ et $a_{8}=X$ ou $a_{8}=Y$ alors on peut réduire $(a_{1},\ldots,a_{n})$ avec $(X,X,X,X,X)$;
\item si $a_{7}=Y$ et $a_{8}=X$ alors on peut réduire $(a_{1},\ldots,a_{n})$ avec $(Y,X,X,Y,X,X,Y,X,X)$;
\item si $a_{7}=Y$ et $a_{8}=Y$ alors soit $a_{9}=Y$ et on peut réduire $(a_{1},\ldots,a_{n})$ avec $(Y,Y,Y,Y,Y)$ soit $a_{9}=X$ et on peut réduire $(a_{1},\ldots,a_{n})$ avec $(Y,Y,X,X,Y,Y,X,X)$;
\end{itemize}
\item si $a_{5}=Y$ et $a_{6}=X$, on s'intéresse aux termes suivants :
\begin{itemize}[label=$\circ$]
\item si $a_{7}=X$ et $a_{8}=X$ alors on peut réduire $(a_{1},\ldots,a_{n})$ avec $(X,X,X,X,X)$;
\item si $a_{7}=Y$ et $a_{8}=X$ alors on peut réduire $(a_{1},\ldots,a_{n})$ avec $(X,Y,X,Y,X,Y)$;
\item si $a_{7}=X$ et $a_{8}=Y$ alors on peut réduire $(a_{1},\ldots,a_{n})$ avec $(X,X,Y,Y,X,X,Y,Y)$;
\item si $a_{7}=Y$ et $a_{8}=Y$ alors soit $a_{9}=Y$ et on peut réduire $(a_{1},\ldots,a_{n})$ avec $(Y,Y,Y,Y,Y)$ soit $a_{9}=X$ et on peut réduire $(a_{1},\ldots,a_{n})$ avec $(Y,X,Y,Y,X,Y,Y,X,Y)$.
\\
\end{itemize}
\end{itemize}

\noindent On suppose maintenant $a_{3}=Y$. On s'intéresse aux termes suivants :
\begin{itemize}
\item si $a_{4}=a_{5}=Y$ alors on peut réduire $(a_{1},\ldots,a_{n})$ avec $(Y,Y,Y,Y,Y)$;
\item si $a_{4}=X$ et $a_{5}=Y$ alors on peut réduire $(a_{1},\ldots,a_{n})$ avec $(Y,X,Y,X,Y,X)$;
\item si $a_{4}=a_{5}=X$, on s'intéresse aux termes suivants :
\begin{itemize}[label=$\circ$]
\item si $a_{6}=X$ et $a_{7}=X$ ou $a_{7}=Y$ alors on peut réduire $(a_{1},\ldots,a_{n})$ avec $(X,X,X,X,X)$;
\item si $a_{6}=Y$ et $a_{7}=X$ alors soit $a_{8}=Y$ et on peut réduire $(a_{1},\ldots,a_{n})$ avec $(Y,X,Y,X,Y,X)$, soit  $a_{8}=X$ et on peut réduire $(a_{1},\ldots,a_{n})$ avec $(X,X,Y,X,X,Y,X,X,Y)$;
\item si $a_{6}=Y$ et $a_{7}=Y$ alors soit $a_{8}=Y$ et on peut réduire $(a_{1},\ldots,a_{n})$ avec $(Y,Y,Y,Y,Y)$, soit $a_{8}=X$ et on peut réduire $(a_{1},\ldots,a_{n})$ avec $(Y,Y,X,X,Y,Y,X,X)$;
\end{itemize}
\item si $a_{4}=Y$ et $a_{5}=X$, on s'intéresse aux termes suivants :
\begin{itemize}[label=$\circ$]
\item si $a_{6}=a_{7}=X$ alors on peut réduire $(a_{1},\ldots,a_{n})$ avec $(X,X,X,X,X)$;
\item si $a_{6}=Y$ et $a_{7}=X$ alors on peut réduire $(a_{1},\ldots,a_{n})$ avec $(X,Y,X,Y,X,Y)$;
\item si $a_{6}=X$ et $a_{7}=Y$ alors on peut réduire $(a_{1},\ldots,a_{n})$ avec $(X,X,Y,Y,X,X,Y,Y)$;
\item si $a_{6}=Y$ et $a_{7}=Y$ alors soit $a_{8}=Y$ et on peut réduire $(a_{1},\ldots,a_{n})$ avec $(Y,Y,Y,Y,Y)$ soit $a_{8}=X$ et on peut réduire $(a_{1},\ldots,a_{n})$ avec $(Y,X,Y,Y,X,Y,Y,X,Y)$.
\\
\end{itemize}
\end{itemize}

\noindent Ainsi, les $\lambda$-quiddités irréductibles sur $\mathbb{F}_{4}$ sont de taille inférieure ou égale à 9 et appartiennent à la liste donnée dans l'énoncé du théorème. 

\end{proof}

Comme $(\mathbb{F}_{4},+)$ est un groupe de cardinal 4 dont les éléments sont d'ordre 1 ou 2, il est isomorphe au groupe $(\mathbb{Z}/2\mathbb{Z} \times \mathbb{Z}/2\mathbb{Z},+)$. Or, les $\lambda$-quiddités irréductibles sur ces deux anneaux sont très différentes. L'existence d'un isomorphisme de groupes entre anneaux commutatifs unitaires ne permet donc pas de transférer d'un anneau à l'autre les connaissances sur les $\lambda$-quiddités.

\subsection{Premières applications}

Dans cette sous-partie, on va donner quelques applications des résultats déjà démontrés et notamment du théorème \ref{26}. Pour commencer, on va utiliser la classification établie dans ce dernier pour connaître les $\lambda$-quiddités irréductibles sur d'autres anneaux. 
\\
\\ \indent Soient $X$ un ensemble et $\mathcal{P}(X)$ l'ensemble des parties de $X$. On munit $\mathcal{P}(X)$ de deux lois de composition internes, $\Delta$ et $\cap$ avec $B \Delta C:=(B \cap C^{c}) \cup (C \cap B^{c})$ (où $(B,C) \in \mathcal{P}(X)^{2}$ et $B^{c}$ désigne le complémentaire de $B$ dans $X$). $\Delta$ est associative et commutative, $\emptyset$ est élément neutre pour $\Delta$ et pour tout $B \in \mathcal{P}(X)$ on a $B \Delta B=\emptyset$. $\cap$ est associative et commutative, $X$ est élément neutre et $\cap$ est distributif par rapport à $\Delta$. En résumé, $(\mathcal{P}(X),\Delta, \cap)$ est un anneau commutatif unitaire. De pus, on dispose des applications suivantes :

\[\begin{array}{ccccc} 
f & : & \mathcal{P}(X) & \longrightarrow & \prod_{x \in X} \mathbb{Z}/2\mathbb{Z} \\
 & & B & \longmapsto & (\overline{b_{x}})_{x \in X}~{\rm avec}~\overline{b_{x}}=\left\{
    \begin{array}{ll}
        \overline{1} & \mbox{si}~x \in B; \\
        \overline{0} & \mbox{sinon}.
\end{array}
\right. \\
\end{array}, \begin{array}{ccccc} 
g & : & \prod_{x \in X} \mathbb{Z}/2\mathbb{Z} & \longrightarrow & \mathcal{P}(X) \\
 & & (\overline{b_{x}})_{x \in X} & \longmapsto & \{x~{\rm tel~que}~\overline{b_{x}}=\overline{1}\}
\end{array}.\]

\noindent $f$ et $g$ sont des bijections réciproques et on vérifie aisément que $f$ est un morphisme d'anneaux unitaires. Aussi, on peut, grâce au corollaire \ref{33}, transférer les résultats connus pour $\prod_{x \in X} \mathbb{Z}/2\mathbb{Z}$. Cela permet notamment d'avoir :

\begin{proposition}
\label{36}

Soit $X=\{a,b\}$. Les $\lambda$-quiddités irréductibles sur $(\mathcal{P}(X),\Delta, \cap)$ sont (à permutations cycliques près): 
\begin{itemize}
\item $(X,X,X)$;
\item $(\{a\},\{b\},\{a\},\{b\})$, $(\emptyset,\emptyset,\emptyset,\emptyset)$, $(\{a\},\emptyset,\{a\},\emptyset)$, $(\{b\},\emptyset,\{b\},\emptyset)$;
\item $(\{a\},\{a\},\{a\},\{a\},\{a\},\{a\})$ et $(\{b\},\{b\},\{b\},\{b\},\{b\},\{b\})$.
\end{itemize}

\end{proposition}

\begin{proof}

C'est une conséquence de la discussion précédente, du corollaire \ref{33} et du théorème \ref{26}.

\end{proof}

De même, $(\mathbb{Z}/2\mathbb{Z}) \times (\mathbb{Z}/2\mathbb{Z})$ est naturellement isomorphe à l'ensemble des matrices diagonales $2 \times 2$ à coefficients dans $\mathbb{Z}/2\mathbb{Z}$. Aussi, on a :

\begin{proposition}
\label{36}

On pose $B=\begin{pmatrix}
    \overline{1} & \overline{0} \\
    \overline{0} &  \overline{0}
   \end{pmatrix}$ et $C=\begin{pmatrix}
    \overline{0} & \overline{0} \\
    \overline{0} &  \overline{1}
   \end{pmatrix}$. Les $\lambda$-quiddités irréductibles sur l'ensemble des matrices diagonales $2 \times 2$ à coefficients dans $\mathbb{Z}/2\mathbb{Z}$ sont (à permutations cycliques près): 
\begin{itemize}
\item $(Id,Id,Id)$;
\item $(B,C,B,C)$, $(B,0,B,0)$, $(C,0,C,0)$, $(0,0,0,0)$;
\item $(B,B,B,B,B,B)$ et $(C,C,C,C,C,C)$.
\end{itemize}

\end{proposition}

Pour avoir des informations pour des matrices de taille plus grande ou pour des ensembles contenant plus d'éléments, il faudrait obtenir plus d'éléments sur les $\lambda$-quiddités irréductibles sur $\prod_{x \in X} \mathbb{Z}/2\mathbb{Z}$. Cela conduit notamment au problème ouvert suivant :

\begin{pro}

Soit $I$ un ensemble. Trouver des conditions nécessaires et suffisantes sur $I$ pour qu'il n'y ait qu'un nombre fini de $\lambda$-quiddités irréductibles sur $\prod_{i \in I} \mathbb{Z}/2\mathbb{Z}$.

\end{pro}

\noindent Dans le cas de $\mathbb{Z}/2\mathbb{Z}$, on dispose d'une description combinatoire élégante des solutions de $(E_{2})$.

\begin{definition}

i) (\cite{M}, Définition 3.1) On appelle décomposition de type (3|4) le découpage d'un polygone convexe $P$ à $n$ sommets par des diagonales ne se coupant qu'aux sommets et tel que les sous-polygones soient des triangles ou des quadrilatères. 
\\
\\ii) (\cite{M}, Définition 3.3) À chaque sommet de $P$ on associe un élément $\overline{c}$ de $\mathbb{Z}/2\mathbb{Z}$ de la façon suivante $$\overline{c}=
\left\{
\begin{array}{ll}\overline{1}, & \hbox{si le nombre de triangles utilisant ce sommet est impair};\\[2pt]
\overline{0}, & \hbox{si le nombre de triangles utilisant ce sommet est pair}.
\end{array}
\right.
$$
On parcourt les sommets, à partir de n'importe lequel d'entre eux, dans le sens horaire ou le sens trigonométrique, pour obtenir le $n$-uplet $(\overline{c_{1}},\ldots,\overline{c_{n}})$. Ce $n$-uplet est la quiddité de la décomposition de type (3|4) de $P$.

\end{definition}

\begin{theorem}[\cite{M}, Théorème 1]
Soit $n \geq 2$.
\\ i) Une solution de $(E_{2})$ de taille $n$ est la quiddité d'une décomposition de type (3|4) d'un polygone convexe à $n$ sommets.
\\ ii) La quiddité d'une décomposition de type (3|4) d'un polygone convexe à $n$ sommets est une solution de $(E_{2})$ de taille $n$.

\end{theorem}

En combinant ce résultat au théorème \ref{25}, on obtient la description combinatoire suivante des $\lambda$-quiddités sur $\prod_{i \in I} \mathbb{Z}/2\mathbb{Z}$. On considère un polygone convexe $P$ à $n$ sommets. Pour chaque $i \in I$, on effectue une décomposition de type (3|4) de $P$, qu'on indexe avec $i$. Pour chaque sommet, on considère la famille $(\overline{a_{i}})_{i \in I}$ où $\overline{a_{i}}$ est l'élément de la quiddité de la décomposition de type (3|4) de $P$ indexée par $i$ utilisant le sommet choisi. En parcourant les sommets dans le sens horaire ou dans le sens trigonométrique, on forme un $n$-uplet d'éléments de $\prod_{i \in I} \mathbb{Z}/2\mathbb{Z}$ qui est une solution de \eqref{a} de taille $n$. Réciproquement, les $\lambda$-quiddités sur $\prod_{i \in I} \mathbb{Z}/2\mathbb{Z}$ de taille $n$ peuvent toute se construire ainsi. Par exemple, considérons le triplet de décompositions de type (3|4) ci-dessous.
$$
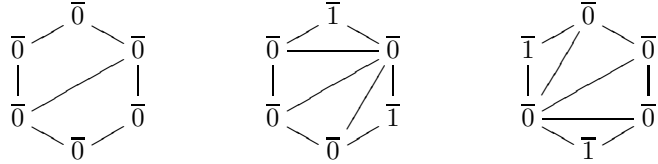
\shorthandoff{; :!?}
\xymatrix @!0 @R=0.45cm @C=0.8cm
{
&\overline{0}\ar@{-}[ld]\ar@{-}[rd]&
\\
\overline{0}\ar@{-}[dd]&&\overline{0}\ar@{-}[dd]
\\
\\
\overline{0}\ar@{-}[rruu]&& \overline{0}
\\
&\overline{0}\ar@{-}[lu]\ar@{-}[ru]&
}
\qquad
\qquad
\xymatrix @!0 @R=0.45cm @C=0.8cm
{
&\overline{1}\ar@{-}[ld]\ar@{-}[rd]&
\\
\overline{0}\ar@{-}[dd]\ar@{-}[rr]&&\overline{0}\ar@{-}[dd]
\\
\\
\overline{0}\ar@{-}[rruu]&& \overline{1}
\\
&\overline{0}\ar@{-}[ruuu]\ar@{-}[lu]\ar@{-}[ru]&
}\qquad
\qquad
\xymatrix @!0 @R=0.45cm @C=0.8cm
{
&\overline{0}\ar@{-}[ld]\ar@{-}[rd]&
\\
\overline{1}\ar@{-}[dd]&&\overline{0}\ar@{-}[dd]
\\
\\
\overline{0}\ar@{-}[ruuu]\ar@{-}[rr]\ar@{-}[rruu]&& \overline{0}
\\
&\overline{1}\ar@{-}[lu]\ar@{-}[ru]&
}
$$

\noindent On obtient (à équivalence près) la solution de \eqref{a} sur $(\mathbb{Z}/2\mathbb{Z}) \times (\mathbb{Z}/2\mathbb{Z}) \times (\mathbb{Z}/2\mathbb{Z})$ suivante : \[((\overline{0},\overline{1},\overline{0}), (\overline{0},\overline{0},\overline{1}), (\overline{0},\overline{0},\overline{0}), (\overline{0},\overline{0},\overline{1}), (\overline{0},\overline{1},\overline{0}), (\overline{0},\overline{0},\overline{0})).\]

Par ailleurs, on peut également utiliser les descriptions combinatoires pour établir la réductibilité d'une solution. Une $\lambda$-quiddité sur $\prod_{i \in I} \mathbb{Z}/2\mathbb{Z}$ est réductible si et seulement s'il existe une famille de décompositions de type (3|4) représentant la solution telle que toutes les décompositions possèdent une diagonale commune. Ceci est, par exemple, le cas du triplet de décomposition ci-dessus. Donc, la solution $((\overline{0},\overline{1},\overline{0}), (\overline{0},\overline{0},\overline{1}), (\overline{0},\overline{0},\overline{0}), (\overline{0},\overline{0},\overline{1}), (\overline{0},\overline{1},\overline{0}), (\overline{0},\overline{0},\overline{0}))$ est réductible.
\\
\\ \indent Toutefois, il faut noter que deux familles de décomposition de type (3|4) différentes peuvent donner la même solution et qu'il est possible d'en avoir certaines qui vérifient la condition de réductibilité alors que d'autres ne la vérifient pas. Aussi, il convient de manier cette interprétation géométrique de la réductibilité avec précaution. Par exemple, $((\overline{0},\overline{0}),(\overline{0},\overline{0}),(\overline{0},\overline{0}), (\overline{0},\overline{0}),(\overline{0},\overline{0}),(\overline{0},\overline{0}),(\overline{0},\overline{0}),(\overline{0},\overline{0}))$ est réductible alors que les deux décompositions de type (3|4) suivantes n'ont pas de diagonale commune.

$$
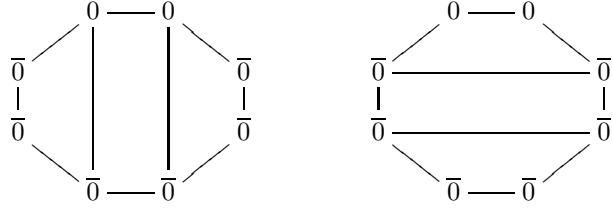
\shorthandoff{; :!?}
\xymatrix @!0 @R=0.40cm @C=0.5cm
{
&&\overline{0}\ar@{-}[lldd]\ar@{-}[dddddd] \ar@{-}[rr]&&\overline{0}\ar@{-}[rrdd]\ar@{-}[dddddd]&
\\
&&&
\\
\overline{0}\ar@{-}[dd]&&&&&& \overline{0} \ar@{-}[dd]
\\
&&&&
\\
\overline{0}&&&&&& \overline{0}
\\
&&&
\\
&&\overline{0} \ar@{-}[rr]\ar@{-}[lluu] &&\overline{0} \ar@{-}[rruu]
}
\qquad
\qquad
\xymatrix @!0 @R=0.40cm @C=0.5cm
{
&&\overline{0}\ar@{-}[lldd] \ar@{-}[rr]&&\overline{0}\ar@{-}[rrdd]&
\\
&&&
\\
\overline{0}\ar@{-}[rrrrrr]\ar@{-}[dd]&&&&&& \overline{0} \ar@{-}[dd]
\\
&&&&
\\
\overline{0}\ar@{-}[rrrrrr]&&&&&& \overline{0}
\\
&&&
\\
&&\overline{0} \ar@{-}[rr]\ar@{-}[lluu] &&\overline{0} \ar@{-}[rruu]
}
$$

Grâce à la description combinatoire de l'opération $\oplus$ (voir \cite{C} section 4), on pourrait généraliser ces descriptions combinatoires à tous les produits directs d'anneaux commutatifs unitaires. Cependant, la complexité des descriptions combinatoires des solutions sur chaque anneau (qui augmente au fur et à mesure que le nombre de solutions irréductibles augmente) et l'ajout de la condition du $\epsilon$ commun présente dans le théorème \ref{25} diminue grandement l'intérêt qu'aurait ces descriptions combinatoires. Aussi, on préfère se centrer ici sur le cas plus simple et plus élégant des $\lambda$-quiddités sur $\prod_{i \in I} \mathbb{Z}/2\mathbb{Z}$.

\section{Applications du théorème \ref{25}}
\label{appli}

Dans cette partie, on va utiliser le théorème \ref{25} pour obtenir des informations sur les $\lambda$-quiddités irréductibles lorsque l'on se place sur divers produits directs d'anneaux commutatifs unitaires.

\subsection{$\lambda$-quiddités sur quelques anneaux du type $\mathbb{Z}/n\mathbb{Z} \times \mathbb{Z}/m \mathbb{Z}$}

L'objectif de cette sous-section est d'étudier en détail les solutions irréductibles de \eqref{a} pour quelques exemples d'anneaux construits comme le produit direct de deux anneaux d'entiers modulaires. Sur ce sujet, on dispose déjà de la description complète dans le cas de $\mathbb{Z}/2\mathbb{Z} \times \mathbb{Z}/2\mathbb{Z}$ (voir Théorème \ref{26}). Ici, on souhaite aller plus loin en considérant d'autres cas. Dans cette optique, il semble intéressant de rappeler le théorème décrivant les $\lambda$-quiddités irréductibles sur $\mathbb{Z}/N\mathbb{Z}$ dans le cas des petites valeurs de $N$.

\begin{theorem}[\cite{M2} Théorème 2.5]
\label{41}

i)Les solutions irréductibles de $(E_{2})$ sont $(\overline{1},\overline{1},\overline{1})$ et $(\overline{0},\overline{0},\overline{0},\overline{0})$.
\\
\\ii)Les solutions irréductibles de $(E_{3})$ sont $(\overline{1},\overline{1},\overline{1})$, $(\overline{-1},\overline{-1},\overline{-1})$ et $(\overline{0},\overline{0},\overline{0},\overline{0})$.
\\
\\iii)Les solutions irréductibles de $(E_{4})$ sont $(\overline{1},\overline{1},\overline{1})$, $(\overline{-1},\overline{-1},\overline{-1})$, $(\overline{0},\overline{0},\overline{0},\overline{0})$, $(\overline{0},\overline{2},\overline{0},\overline{2})$ ,$(\overline{2},\overline{0},\overline{2},\overline{0})$ et $(\overline{2},\overline{2},\overline{2},\overline{2})$.
\\
\\iv)Les solutions irréductibles de $(E_{5})$ sont (à permutations cycliques près) $(\overline{1},\overline{1},\overline{1})$, $(\overline{-1},\overline{-1},\overline{-1})$, $(\overline{0},\overline{0},\overline{0},\overline{0})$, $(\overline{0},\overline{2},\overline{0},\overline{3})$, $(\overline{2},\overline{2},\overline{2},\overline{2},\overline{2})$, $(\overline{3},\overline{3},\overline{3},\overline{3},\overline{3})$, $(\overline{3},\overline{2},\overline{2},\overline{3},\overline{2},\overline{2})$, $(\overline{2},\overline{3},\overline{3},\overline{2},\overline{3},\overline{3})$, $(\overline{2},\overline{3},\overline{2},\overline{3},\overline{2},\overline{3})$.
\\
\\v)Les solutions irréductibles de $(E_{6})$ sont (à permutations cycliques près) $(\overline{1},\overline{1},\overline{1})$, $(\overline{-1},\overline{-1},\overline{-1})$, $(\overline{0},\overline{0},\overline{0},\overline{0})$, $(\overline{2},\overline{4},\overline{2},\overline{4})$, $(\overline{2},\overline{3},\overline{4},\overline{3})$, $(\overline{0},\overline{2},\overline{0},\overline{4})$, $(\overline{0},\overline{3},\overline{0},\overline{3})$, $(\overline{2},\overline{2},\overline{2},\overline{2},\overline{2},\overline{2})$, $(\overline{3},\overline{3},\overline{3},\overline{3},\overline{3},\overline{3})$, $(\overline{4},\overline{4},\overline{4},\overline{4},\overline{4},\overline{4})$.

\end{theorem}

Par ailleurs, on dispose également de formules permettant de connaître le nombre de $\lambda$-quiddités sur $\mathbb{Z}/N\mathbb{Z}$, lorsque $N$ est premier ou lorsque $N=4m$ avec $m$ sans facteur carré (voir \cite{CM} Théorèmes 1.1 et 1.3). En les combinant au théorème \ref{25}, on peut obtenir le nombre de $\lambda$-quiddités sur certains produits directs d'anneaux $\mathbb{Z}/N\mathbb{Z}$.

\medskip        

\subsubsection{Cas de $\mathbb{Z}/2\mathbb{Z} \times \mathbb{Z}/3 \mathbb{Z}$}

Pour commencer, on va redémontrer, à l'aide du théorème \ref{25}, le point v) du théorème précédent, que l'on avait à l'origine obtenu de façon directe, en considérant l'anneau $\mathbb{Z}/2\mathbb{Z} \times \mathbb{Z}/3 \mathbb{Z}$. Plus précisément, on va démontrer le résultat ci-dessous. Pour cela, on notera, pour $a \in \mathbb{Z}$, $\overline{a}:=a+2\mathbb{Z}$ et $\overline{\overline{a}}:=a+3\mathbb{Z}$.

\begin{proposition}
\label{42}

Les $\lambda$-quiddités irréductibles sur $(\mathbb{Z}/2\mathbb{Z}) \times (\mathbb{Z}/3\mathbb{Z})$ sont (à permutations cycliques près) :
\begin{itemize}
\item $((\overline{1},\overline{\overline{1}}),(\overline{1},\overline{\overline{1}}),(\overline{1},\overline{\overline{1}}))$, $((\overline{1},\overline{\overline{-1}}),(\overline{1},\overline{\overline{-1}}),(\overline{1},\overline{\overline{-1}}))$;
\item $((\overline{0},\overline{\overline{0}}),(\overline{0},\overline{\overline{0}}),(\overline{0},\overline{\overline{0}}),(\overline{0},\overline{\overline{0}})), ((\overline{0},\overline{\overline{0}}),(\overline{0},\overline{\overline{1}}),(\overline{0},\overline{\overline{0}}),(\overline{0},\overline{\overline{-1}})), ((\overline{0},\overline{\overline{0}}),(\overline{1},\overline{\overline{0}}),(\overline{0},\overline{\overline{0}}),(\overline{1},\overline{\overline{0}})), 
\\((\overline{0},\overline{\overline{1}}),(\overline{0},\overline{\overline{-1}}),(\overline{0},\overline{\overline{1}}),(\overline{0},\overline{\overline{-1}})), ((\overline{1},\overline{\overline{0}}),(\overline{0},\overline{\overline{-1}}),(\overline{1},\overline{\overline{0}}),(\overline{0},\overline{\overline{1}}))$;
\item $((\overline{1},\overline{\overline{0}}),(\overline{1},\overline{\overline{0}}),(\overline{1},\overline{\overline{0}}),(\overline{1},\overline{\overline{0}}),(\overline{1},\overline{\overline{0}}),(\overline{1},\overline{\overline{0}})), ((\overline{0},\overline{\overline{1}}),(\overline{0},\overline{\overline{1}}),(\overline{0},\overline{\overline{1}}),(\overline{0},\overline{\overline{1}}),(\overline{0},\overline{\overline{1}}),(\overline{0},\overline{\overline{1}})), 
\\((\overline{0},\overline{\overline{-1}}),(\overline{0},\overline{\overline{-1}}),(\overline{0},\overline{\overline{-1}}),(\overline{0},\overline{\overline{-1}}),(\overline{0},\overline{\overline{-1}}),(\overline{0},\overline{\overline{-1}}))$.
\end{itemize}

\end{proposition}

Une fois ce résultat démontré, il suffit, pour retrouver le point v) du théorème \ref{41}, d'appliquer le corollaire \ref{33} et le lemme chinois. Dans notre cas, si $\varphi: (\mathbb{Z}/2\mathbb{Z}) \times (\mathbb{Z}/3\mathbb{Z}) \longrightarrow (\mathbb{Z}/6\mathbb{Z})$ est l'isomorphisme du lemme chinois, $\varphi((\overline{0},\overline{\overline{0}})=0+6\mathbb{Z}, \varphi((\overline{0},\overline{\overline{1}})=4+6\mathbb{Z}, \varphi((\overline{0},\overline{\overline{-1}})=2+6\mathbb{Z}, \varphi((\overline{1},\overline{\overline{0}})=3+6\mathbb{Z}, \varphi((\overline{1},\overline{\overline{1}})=1+6\mathbb{Z}$ et $\varphi((\overline{1},\overline{\overline{-1}})=5+6\mathbb{Z}$.

\begin{proof}

Avant de commencer, notons que $\mathbb{Z}/2\mathbb{Z}$ est de caractéristique 2 et que donc la condition de l'existence d'un $\epsilon$ commun présente dans le théorème \ref{25} est automatiquement vérifiée pour l'anneau considéré ici. Par ailleurs, la preuve ci-dessous étant très proche de celle du théorème \ref{26}, on omettra certains détails.
\\
\\À la lueur du théorème \ref{25}, on voit qu'il suffit de trouver les couples formés d'une solution de $(E_{2})$ et d'une solution de $(E_{3})$ qui ne sont pas simultanément réductibles. En utilisant le théorème \ref{25} et le lemme \ref{35}, on constate que les solutions irréductibles de \eqref{a} sur $\mathbb{Z}/2\mathbb{Z} \times \mathbb{Z}/3\mathbb{Z}$ de taille 3 et 4 sont celles données dans l'énoncé.
\\
\\Soient $n \geq 5$, $(\overline{a_{1}},\ldots,\overline{a_{n}})$ une $\lambda$-quiddité sur $\mathbb{Z}/2\mathbb{Z}$ et $(\overline{\overline{b_{1}}},\ldots,\overline{\overline{b_{n}}})$ une $\lambda$-quiddité sur $\mathbb{Z}/3\mathbb{Z}$. On suppose qu'elles ne sont pas simultanément réductibles.
\\
\\On suppose pour commencer qu'il existe $1 \leq j \leq n$ tel que $\overline{a_{j}}=\overline{1}$. Si $\overline{\overline{b_{j}}}=\overline{\overline{1}}$  ou $\overline{\overline{-1}}$ alors les deux solutions sont simultanément réductibles. Donc, $\overline{\overline{b_{j}}}=\overline{\overline{0}}$. Si $\overline{a_{j+1}}=\overline{0}$ les deux solutions sont simultanément réductibles. Ainsi, $\overline{a_{j+1}}=\overline{1}$. Si $\overline{\overline{b_{j+1}}}=\overline{\overline{1}}$ ou $\overline{\overline{-1}}$ alors les deux solutions sont simultanément réductibles. Donc, $\overline{\overline{b_{j+1}}}=\overline{\overline{0}}$. En continuant ainsi, on voit que tous les $\overline{a_{j}}$ sont égaux à $\overline{1}$ et que tous les $\overline{\overline{b_{j}}}$ sont égaux à $\overline{\overline{0}}$. Comme $M_{1}(\overline{1})$ est d'ordre 3 dans $PSL_{2}(\mathbb{Z}/2\mathbb{Z})$ et $M_{1}(\overline{\overline{0}})$ est d'ordre 2 dans $PSL_{2}(\mathbb{Z}/3\mathbb{Z})$, $n$ est divisible par 6. De plus, si $n \geq 12$, les deux solutions sont simultanément réductibles. En revanche, si $n=6$ les deux solutions ne sont pas simultanément réductibles. En effet, pour réduire simultanément les deux solutions, il nous faudrait une solution de taille 3 de la forme $(\overline{\overline{x}},\overline{\overline{0}},\overline{\overline{y}})$ ou une solution de taille 4 de la forme $(\overline{x},\overline{1},\overline{1},\overline{y})$, ce qui n'existe pas.
\\
\\Si, pour tout $1 \leq j \leq n$, $\overline{a_{j}}=\overline{0}$. S'il existe $1 \leq k \leq n$ tel que $\overline{\overline{b_{k}}}=\overline{\overline{0}}$ alors les deux solutions sont simultanément réductibles. Donc, tous les $\overline{\overline{b_{j}}}$ sont égaux à $\pm \overline{\overline{1}}$. S'il existe $1 \leq k \leq n-1$ tel que $\overline{\overline{b_{k}}}=\overline{\overline{\epsilon}}$ et $\overline{\overline{b_{k+1}}}=\overline{\overline{-\epsilon}}$, avec $\epsilon \in \{-1, 1\}$. On a :
\[\left\{
    \begin{array}{ll}
        (\overline{a_{k+2}},\ldots,\overline{a_{n}},\overline{a_{1}},\ldots,\overline{a_{k}},\overline{a_{k+1}})=(\overline{a_{k+2}},\ldots,\overline{a_{n}},\overline{a_{1}},\ldots,\overline{a_{k-1}}) \oplus (\overline{0},\overline{0},\overline{0},\overline{0}); \\
        (\overline{\overline{b_{k+2}}},\ldots,\overline{\overline{b_{n}}},\overline{\overline{b_{1}}},\ldots,\overline{\overline{b_{k}}},\overline{\overline{b_{k+1}}})=(\overline{\overline{b_{k+2}-\epsilon}},\ldots,\overline{\overline{b_{n}}},\overline{\overline{b_{1}}},\ldots,\overline{\overline{b_{k-1}+\epsilon}}) \oplus (\overline{\overline{-\epsilon}},\overline{\overline{\epsilon}},\overline{\overline{-\epsilon}},\overline{\overline{\epsilon}}).
    \end{array}
\right. \\ \]

\noindent Donc, pour tout $1 \leq j \leq n$ $\overline{\overline{b_{j}}}=\overline{\overline{1}}$, ou, pour tout $1 \leq j \leq n$ $\overline{\overline{b_{j}}}=\overline{\overline{-1}}$. En procédant comme dans la preuve du théorème \ref{26}, on aboutit à $(\overline{a_{1}},\ldots,\overline{a_{n}})=(\overline{0},\overline{0},\overline{0},\overline{0},\overline{0},\overline{0})$ et $(\overline{\overline{b_{1}}},\ldots,\overline{\overline{b_{n}}})=\pm (\overline{\overline{1}},\overline{\overline{1}},\overline{\overline{1}},\overline{\overline{1}},\overline{\overline{1}},\overline{\overline{1}})$.

\end{proof}

\subsubsection{Cas de $\mathbb{Z}/2\mathbb{Z} \times \mathbb{Z}/4 \mathbb{Z}$}  

On va maintenant démontrer le résultat ci-dessous :

\begin{proposition}
\label{43}

Les $\lambda$-quiddités irréductibles sur $(\mathbb{Z}/2\mathbb{Z}) \times (\mathbb{Z}/4\mathbb{Z})$ sont de taille 3, 4 ou 6.

\end{proposition} 

\begin{proof}

On notera, pour $a \in \mathbb{Z}$, $\overline{a}:=a+2\mathbb{Z}$ et $\overline{\overline{a}}:=a+4\mathbb{Z}$. Soient $n \geq 5$, $(\overline{a_{1}},\ldots,\overline{a_{n}})$ une $\lambda$-quiddité sur $\mathbb{Z}/2\mathbb{Z}$ et $(\overline{\overline{b_{1}}},\ldots,\overline{\overline{b_{n}}})$ une $\lambda$-quiddité sur $\mathbb{Z}/4\mathbb{Z}$. On suppose qu'elles ne sont pas simultanément réductibles.
\\
\\On suppose pour commencer qu'il existe $1 \leq j \leq n$ tel que $\overline{a_{j}}=\overline{1}$. Si $\overline{\overline{b_{j}}}=\overline{\overline{1}}$  ou $\overline{\overline{-1}}$ alors les deux solutions sont simultanément réductibles. Donc, $\overline{\overline{b_{j}}}=\overline{\overline{0}}$ ou $\overline{\overline{b_{j}}}=\overline{\overline{2}}$. Si $\overline{a_{j+1}}=\overline{0}$ les deux solutions sont simultanément réductibles. En effet, si $\overline{\overline{b_{j}}}=\overline{\overline{0}}$ on a :
\[\left\{
    \begin{array}{ll}
        (\overline{a_{j+2}},\ldots,\overline{a_{n}},\overline{a_{1}},\ldots,\overline{a_{j}},\overline{a_{j+1}})=(\overline{a_{j+2}+1},\ldots,\overline{a_{n}},\overline{a_{1}},\ldots,\overline{a_{j-1}}) \oplus (\overline{0},\overline{1},\overline{0},\overline{1}); \\
        (\overline{\overline{b_{j+2}}},\ldots,\overline{\overline{b_{n}}},\overline{\overline{b_{1}}},\ldots,\overline{\overline{b_{j}}},\overline{\overline{b_{j+1}}})=(\overline{\overline{b_{j+2}}},\ldots,\overline{\overline{b_{n}}},\overline{\overline{b_{1}}},\ldots,\overline{\overline{b_{j-1}+b_{j+1}}}) \oplus (\overline{\overline{-b_{j+1}}},\overline{\overline{0}},\overline{\overline{b_{j+1}}},\overline{\overline{0}}).
    \end{array}
\right. \\ \]

\noindent Si $\overline{\overline{b_{j}}}=\overline{\overline{2}}$ on a :
\[\left\{
    \begin{array}{ll}
        (\overline{a_{j+2}},\ldots,\overline{a_{n}},\overline{a_{1}},\ldots,\overline{a_{j}},\overline{a_{j+1}})=(\overline{a_{j+2}+1},\ldots,\overline{a_{n}},\overline{a_{1}},\ldots,\overline{a_{j-1}}) \oplus (\overline{0},\overline{1},\overline{0},\overline{1}); \\
        (\overline{\overline{b_{j+2}}},\ldots,\overline{\overline{b_{n}}},\overline{\overline{b_{1}}},\ldots,\overline{\overline{b_{j}}},\overline{\overline{b_{j+1}}})=(\overline{\overline{b_{j+2}-2}},\ldots,\overline{\overline{b_{n}}},\overline{\overline{b_{1}}},\ldots,\overline{\overline{b_{j-1}-b_{j+1}}}) \oplus (\overline{\overline{b_{j+1}}},\overline{\overline{2}},\overline{\overline{b_{j+1}}},\overline{\overline{2}}).
    \end{array}
\right. \\ \]

\noindent Ainsi, $\overline{a_{j+1}}=\overline{1}$. Si $\overline{\overline{b_{j+1}}}=\overline{\overline{1}}$ ou $\overline{\overline{-1}}$ alors les deux solutions sont simultanément réductibles. Donc, $\overline{\overline{b_{j+1}}}=\overline{\overline{0}}$ ou $\overline{\overline{b_{j}}}=\overline{\overline{2}}$. En continuant ainsi, on voit que tous les $\overline{a_{j}}$ sont égaux à $\overline{1}$ et que tous les $\overline{\overline{b_{j}}}$ sont égaux à $\overline{\overline{0}}$ ou $\overline{\overline{2}}$. 
\\
\\Or, une solution de $(E_{2})$ ne contenant que des $\overline{1}$ a une taille qui est un multiple de 3 et une solution de $(E_{4})$ ne contenant que $\overline{\overline{0}}$ ou $\overline{\overline{2}}$ a une taille paire. Donc, $n$ est divisible par 6. Si $n=6$ les deux solutions ne sont pas simultanément réductibles. En effet, pour réduire simultanément les deux solutions, il nous faudrait une solution de taille 3 de la forme $(\overline{\overline{x}},\overline{\overline{0}},\overline{\overline{y}})$ ou $(\overline{\overline{x}},\overline{\overline{2}},\overline{\overline{y}})$ ou une solution de taille 4 de la forme $(\overline{x},\overline{1},\overline{1},\overline{y})$, ce qui n'existe pas.
\\
\\Si $n \geq 12$. $(E_{4})$ a 16 solutions de taille 6 ne contenant que $\overline{\overline{0}}$ ou $\overline{\overline{2}}$ : $(\overline{\overline{0}},\overline{\overline{0}},\overline{\overline{0}},\overline{\overline{0}},\overline{\overline{0}},\overline{\overline{0}})$, $(\overline{\overline{0}},\overline{\overline{0}},\overline{\overline{0}},\overline{\overline{2}},\overline{\overline{0}},\overline{\overline{2}})$, $(\overline{\overline{0}},\overline{\overline{0}},\overline{\overline{2}},\overline{\overline{0}},\overline{\overline{2}},\overline{\overline{0}})$, $(\overline{\overline{0}},\overline{\overline{0}},\overline{\overline{2}},\overline{\overline{2}},\overline{\overline{2}},\overline{\overline{2}})$, $(\overline{\overline{0}},\overline{\overline{2}},\overline{\overline{0}},\overline{\overline{0}},\overline{\overline{0}},\overline{\overline{2}})$, $(\overline{\overline{0}},\overline{\overline{2}},\overline{\overline{0}},\overline{\overline{2}},\overline{\overline{0}},\overline{\overline{0}})$, $(\overline{\overline{0}},\overline{\overline{2}},\overline{\overline{2}},\overline{\overline{0}},\overline{\overline{2}},\overline{\overline{2}})$, $(\overline{\overline{0}},\overline{\overline{2}},\overline{\overline{2}},\overline{\overline{2}},\overline{\overline{2}},\overline{\overline{0}})$, $(\overline{\overline{2}},\overline{\overline{0}},\overline{\overline{0}},\overline{\overline{0}},\overline{\overline{2}},\overline{\overline{0}})$, $(\overline{\overline{2}},\overline{\overline{0}},\overline{\overline{0}},\overline{\overline{2}},\overline{\overline{2}},\overline{\overline{2}})$, $(\overline{\overline{2}},\overline{\overline{0}},\overline{\overline{2}},\overline{\overline{0}},\overline{\overline{0}},\overline{\overline{0}})$, $(\overline{\overline{2}},\overline{\overline{0}},\overline{\overline{2}},\overline{\overline{2}},\overline{\overline{0}},\overline{\overline{2}})$, $(\overline{\overline{2}},\overline{\overline{2}},\overline{\overline{0}},\overline{\overline{0}},\overline{\overline{2}},\overline{\overline{2}})$, $(\overline{\overline{2}},\overline{\overline{2}},\overline{\overline{0}},\overline{\overline{2}},\overline{\overline{2}},\overline{\overline{0}})$, $(\overline{\overline{2}},\overline{\overline{2}},\overline{\overline{2}},\overline{\overline{0}},\overline{\overline{0}},\overline{\overline{2}})$, $(\overline{\overline{2}},\overline{\overline{2}},\overline{\overline{2}},\overline{\overline{2}},\overline{\overline{0}},\overline{\overline{0}})$. 
\\
\\On peut réduire $(\overline{\overline{b_{1}}},\ldots,\overline{\overline{b_{n}}})$ avec une des solutions ci-dessus, puisque pour chaque 4-uplet $(\overline{\overline{x}},\overline{\overline{y}},\overline{\overline{z}},\overline{\overline{t}})$ ne contenant que $\overline{\overline{0}}$ ou $\overline{\overline{2}}$ il existe une solution de la forme $(\overline{\overline{u}},\overline{\overline{x}},\overline{\overline{y}},\overline{\overline{z}},\overline{\overline{t}},\overline{\overline{v}})$. Donc, il existe $\overline{\overline{u}}, \overline{\overline{v}} \in \{\overline{\overline{0}}, \overline{\overline{2}}\}$ tels que $(\overline{\overline{u}},\overline{\overline{b_{n-3}}},\overline{\overline{b_{n-2}}},\overline{\overline{b_{n-1}}},\overline{\overline{b_{n}}},\overline{\overline{v}})$ est une solution de $(E_{4})$. On a :
\[\left\{
    \begin{array}{ll}
        (\overline{a_{1}},\ldots,\overline{a_{n-4}},\overline{a_{n-3}},\overline{a_{n-2}},\overline{a_{n-1}},\overline{a_{n}})=(\overline{a_{1}+1},\overline{a_{2}},\ldots,\overline{a_{n-5}},\overline{a_{n-4}+1}) \oplus (\overline{1},\overline{1},\overline{1},\overline{1},\overline{1},\overline{1}); \\
        (\overline{\overline{b_{1}}},\ldots,\overline{\overline{b_{n-4}}},\overline{\overline{b_{n-3}}},\overline{\overline{b_{n-2}}},\overline{\overline{b_{n-1}}},\overline{\overline{b_{n}}})= (\overline{\overline{b_{1}-v}},\overline{\overline{b_{2}}},\ldots,\overline{\overline{b_{n-5}}},\overline{\overline{b_{n-4}-u}}) \oplus (\overline{\overline{u}},\overline{\overline{b_{n-3}}},\overline{\overline{b_{n-2}}},\overline{\overline{b_{n-1}}},\overline{\overline{b_{n}}},\overline{\overline{v}}).
    \end{array}
\right. \\ \]

\noindent On peut donc réduire simultanément $(\overline{a_{1}},\ldots,\overline{a_{n}})$ et $(\overline{\overline{b_{1}}},\ldots,\overline{\overline{b_{n}}})$. Donc, $n=6$.
\\
\\Si, pour tout $1 \leq j \leq n$, $\overline{a_{j}}=\overline{0}$. S'il existe $1 \leq k \leq n$ tel que $\overline{\overline{b_{k}}}=\overline{\overline{0}}$ ou $\overline{\overline{2}}$ alors les deux solutions sont simultanément réductibles. En effet, si $\overline{\overline{b_{k}}}=\overline{\overline{0}}$, on a :
\[\left\{
    \begin{array}{ll}
        (\overline{a_{k+2}},\ldots,\overline{a_{n}},\overline{a_{1}},\ldots,\overline{a_{k}},\overline{a_{k+1}})=(\overline{a_{k+2}},\ldots,\overline{a_{n}},\overline{a_{1}},\ldots,\overline{a_{k-1}}) \oplus (\overline{0},\overline{0},\overline{0},\overline{0}); \\
        (\overline{\overline{b_{k+2}}},\ldots,\overline{\overline{b_{n}}},\overline{\overline{b_{1}}},\ldots,\overline{\overline{b_{k}}},\overline{\overline{b_{k+1}}})=(\overline{\overline{b_{k+2}}},\ldots,\overline{\overline{b_{n}}},\overline{\overline{b_{1}}},\ldots,\overline{\overline{b_{k-1}+b_{k+1}}}) \oplus (\overline{\overline{-b_{k+1}}},\overline{\overline{0}},\overline{\overline{b_{k+1}}},\overline{\overline{0}}).
    \end{array}
\right. \\ \]

\noindent Si $\overline{\overline{b_{k}}}=\overline{\overline{2}}$ on a :
\[\left\{
    \begin{array}{ll}
        (\overline{a_{k+2}},\ldots,\overline{a_{n}},\overline{a_{1}},\ldots,\overline{a_{k}},\overline{a_{k+1}})=(\overline{a_{k+2}},\ldots,\overline{a_{n}},\overline{a_{1}},\ldots,\overline{a_{k-1}}) \oplus (\overline{0},\overline{0},\overline{0},\overline{0}); \\
        (\overline{\overline{b_{k+2}}},\ldots,\overline{\overline{b_{n}}},\overline{\overline{b_{1}}},\ldots,\overline{\overline{b_{k}}},\overline{\overline{b_{k+1}}})=(\overline{\overline{b_{k+2}-2}},\ldots,\overline{\overline{b_{n}}},\overline{\overline{b_{1}}},\ldots,\overline{\overline{b_{k-1}-b_{k+1}}}) \oplus (\overline{\overline{b_{k+1}}},\overline{\overline{2}},\overline{\overline{b_{k+1}}},\overline{\overline{2}}).
    \end{array}
\right. \\ \] 

\noindent Ainsi, tous les $\overline{\overline{b_{j}}}$ sont égaux à $\pm \overline{\overline{1}}$. Comme il n'existe pas de solution de $(E_{4})$ de taille 5 ne contenant que $\pm \overline{\overline{1}}$, on a $n \geq 6$. 
\\
\\Supposons $n \geq 7$. $(E_{4})$ a 16 solutions de taille 6 ne contenant que $\pm \overline{\overline{1}}$ : $(\overline{\overline{1}},\overline{\overline{1}},\overline{\overline{1}},\overline{\overline{1}},\overline{\overline{1}},\overline{\overline{1}})$, $(\overline{\overline{1}},\overline{\overline{1}},\overline{\overline{1}},\overline{\overline{-1}},\overline{\overline{-1}},\overline{\overline{-1}})$, $(\overline{\overline{1}},\overline{\overline{1}},\overline{\overline{-1}},\overline{\overline{1}},\overline{\overline{1}},\overline{\overline{-1}})$, $(\overline{\overline{1}},\overline{\overline{1}},\overline{\overline{-1}},\overline{\overline{-1}},\overline{\overline{-1}},\overline{\overline{1}})$, $(\overline{\overline{1}},\overline{\overline{-1}},\overline{\overline{1}},\overline{\overline{1}},\overline{\overline{-1}},\overline{\overline{1}})$, $(\overline{\overline{1}},\overline{\overline{-1}},\overline{\overline{1}},\overline{\overline{-1}},\overline{\overline{1}},\overline{\overline{-1}})$, $(\overline{\overline{-1}},\overline{\overline{-1}},\overline{\overline{-1}},\overline{\overline{-1}},\overline{\overline{-1}},\overline{\overline{-1}})$,
\\$(\overline{\overline{1}},\overline{\overline{-1}},\overline{\overline{-1}},\overline{\overline{-1}},\overline{\overline{1}},\overline{\overline{1}})$, $(\overline{\overline{-1}},\overline{\overline{1}},\overline{\overline{1}},\overline{\overline{1}},\overline{\overline{-1}},\overline{\overline{-1}})$, $(\overline{\overline{-1}},\overline{\overline{1}},\overline{\overline{-1}},\overline{\overline{1}},\overline{\overline{-1}},\overline{\overline{1}})$, $(\overline{\overline{-1}},\overline{\overline{1}},\overline{\overline{-1}},\overline{\overline{-1}},\overline{\overline{1}},\overline{\overline{-1}})$, $(\overline{\overline{-1}},\overline{\overline{-1}},\overline{\overline{1}},\overline{\overline{-1}},\overline{\overline{-1}},\overline{\overline{1}})$, 
\\$(\overline{\overline{-1}},\overline{\overline{-1}},\overline{\overline{-1}},\overline{\overline{1}},\overline{\overline{1}},\overline{\overline{1}})$, $(\overline{\overline{-1}},\overline{\overline{-1}},\overline{\overline{1}},\overline{\overline{1}},\overline{\overline{1}},\overline{\overline{-1}})$, $(\overline{\overline{1}},\overline{\overline{-1}},\overline{\overline{-1}},\overline{\overline{1}},\overline{\overline{-1}},\overline{\overline{-1}})$, $(\overline{\overline{-1}},\overline{\overline{1}},\overline{\overline{1}},\overline{\overline{-1}},\overline{\overline{1}},\overline{\overline{1}})$.
\\
\\On peut réduire $(\overline{\overline{b_{1}}},\ldots,\overline{\overline{b_{n}}})$ avec une des solutions ci-dessus, puisque pour chaque 4-uplet $(\overline{\overline{x}},\overline{\overline{y}},\overline{\overline{z}},\overline{\overline{t}})$ ne contenant que $\pm \overline{\overline{1}}$ il existe une solution de la forme $(\overline{\overline{u}},\overline{\overline{x}},\overline{\overline{y}},\overline{\overline{z}},\overline{\overline{t}},\overline{\overline{v}})$. Donc, il existe $\overline{\overline{u}}, \overline{\overline{v}} \in \{\overline{\overline{1}}, \overline{\overline{-1}}\}$ tels que $(\overline{\overline{u}},\overline{\overline{b_{n-3}}},\overline{\overline{b_{n-2}}},\overline{\overline{b_{n-1}}},\overline{\overline{b_{n}}},\overline{\overline{v}})$ est une solution de $(E_{4})$. On a :
\[\left\{
    \begin{array}{ll}
        (\overline{a_{1}},\ldots,\overline{a_{n-4}},\overline{a_{n-3}},\overline{a_{n-2}},\overline{a_{n-1}},\overline{a_{n}})=(\overline{a_{1}},\overline{a_{2}},\ldots,\overline{a_{n-5}},\overline{a_{n-4}}) \oplus (\overline{0},\overline{0},\overline{0},\overline{0},\overline{0},\overline{0}); \\
        (\overline{\overline{b_{1}}},\ldots,\overline{\overline{b_{n-4}}},\overline{\overline{b_{n-3}}},\overline{\overline{b_{n-2}}},\overline{\overline{b_{n-1}}},\overline{\overline{b_{n}}})= (\overline{\overline{b_{1}-v}},\overline{\overline{b_{2}}},\ldots,\overline{\overline{b_{n-5}}},\overline{\overline{b_{n-4}-u}}) \oplus (\overline{\overline{u}},\overline{\overline{b_{n-3}}},\overline{\overline{b_{n-2}}},\overline{\overline{b_{n-1}}},\overline{\overline{b_{n}}},\overline{\overline{v}}).
    \end{array}
\right. \\ \]

\noindent On peut donc réduire simultanément $(\overline{a_{1}},\ldots,\overline{a_{n}})$ et $(\overline{\overline{b_{1}}},\ldots,\overline{\overline{b_{n}}})$ ($n-4 \geq 3$). Donc, $n=6$. Si $n=6$ on a, en procédant comme précédemment, que les deux solutions ne sont pas simultanément réductibles.

\end{proof}

\begin{remark}

{\rm On ne précise pas la liste exhaustive des solutions irréductibles car cette dernière est assez longue. Toutefois, la preuve ci-dessus fournit les constructions à effectuer pour obtenir cette liste. Pour les $\lambda$-quiddités irréductibles de taille 3 et 4, il suffit d'appliquer le théorème \ref{25} et le lemme \ref{35}. Pour les $\lambda$-quiddités irréductibles de taille 6, on construit, grâce au théorème \ref{25}, les solutions avec $(\overline{1},\overline{1},\overline{1},\overline{1},\overline{1},\overline{1})$ et les solutions de $(E_{4})$ de taille 6 ne contenant que $\overline{\overline{0}}$ ou $\overline{\overline{2}}$ données dans la preuve, ou avec $(\overline{0},\overline{0},\overline{0},\overline{0},\overline{0},\overline{0})$ et les solutions de $(E_{4})$ de taille 6 ne contenant que $\pm \overline{\overline{1}}$ données dans la preuve. 
}
\end{remark}

Avant de passer à d'autres cas, on va généraliser les deux résultats sur les solutions de taille 6 de $(E_{4})$ utilisés dans la preuve. Ces derniers ne sont bien entendu pas reliés aux produits directs d'anneaux mais sont néanmoins intéressants à noter du fait de la configuration tout à fait atypique qui y est exposée.

\begin{proposition}
\label{44}

On se place sur $\mathbb{Z}/4\mathbb{Z}$. 
\\
\\i) Soit $n \geq 2$ un entier pair. Soit $(\overline{a_{1}},\ldots,\overline{a_{n}})$ un $n$-uplet ne contenant que $\overline{0}$ et $\overline{2}$. Il existe $(\overline{x},\overline{y}) \in \{\overline{0}, \overline{2}\}^{2}$ tel que $(\overline{x},\overline{a_{1}},\ldots,\overline{a_{n}},\overline{y})$ est une solution de $(E_{4})$. En particulier :
\begin{itemize}
\item pour tout $n$-uplet $(\overline{a_{1}},\ldots,\overline{a_{n}})$ ne contenant que $\overline{0}$ et $\overline{2}$, $K_{n}(\overline{a_{1}},\ldots,\overline{a_{n}})=\pm \overline{1}$;
\item il y a $2^{n-2}$ solutions de $(E_{4})$ de taille $n$ ne contenant que $\overline{0}$ et $\overline{2}$.
\end{itemize}

\noindent ii) Soit $n \equiv 1 [3]$. Soit $(\overline{a_{1}},\ldots,\overline{a_{n}})$ un $n$-uplet ne contenant que $\pm \overline{1}$. Il existe $(\overline{x},\overline{y}) \in \{\overline{-1}, \overline{1}\}^{2}$ tel que $(\overline{x},\overline{a_{1}},\ldots,\overline{a_{n}},\overline{y})$ est une solution de $(E_{4})$. En particulier :
\begin{itemize}
\item pour tout $n \equiv 1 [3]$ et pour tout $n$-uplet $(\overline{a_{1}},\ldots,\overline{a_{n}})$ ne contenant que $\pm \overline{1}$, $K_{n}(\overline{a_{1}},\ldots,\overline{a_{n}})=\pm \overline{1}$;
\item pour tout $n \equiv 0 [3]$, il y a $2^{n-2}$ solutions de $(E_{4})$ de taille $n$ ne contenant que $\overline{1}$ et $\overline{-1}$.
\end{itemize}

\end{proposition}

\begin{proof}

i) On pose $n=2m$ et on procède par récurrence sur $m$. Si $m=1$ alors le résultat est vrai puisque $(\overline{0},\overline{0},\overline{0},\overline{0})$, $(\overline{0},\overline{2},\overline{0},\overline{2})$, $(\overline{2},\overline{0},\overline{2},\overline{0})$ et $(\overline{2},\overline{2},\overline{2},\overline{2})$ sont des solutions de $(E_{4})$.
\\
\\On suppose qu'il existe un $m \in \mathbb{N}^{*}$ tel que pour tout $2m$-uplet $(\overline{a_{1}},\ldots,\overline{a_{2m}})$ ne contenant que $\overline{0}$ et $\overline{2}$ il existe $(\overline{x},\overline{y}) \in \{\overline{0}, \overline{2}\}^{2}$ tel que $(\overline{x},\overline{a_{1}},\ldots,\overline{a_{n}},\overline{y})$ est une solution de $(E_{4})$. Soit $(\overline{a_{1}},\ldots,\overline{a_{2m+2}})$ un $(2m+2)$-uplet ne contenant que $\overline{0}$ et $\overline{2}$. $(\overline{a_{1}},\ldots,\overline{a_{2m-1}},\overline{a_{2m}-a_{2m+2}})$ ne contient que $\overline{0}$ et $\overline{2}$. Il existe $(\overline{u},\overline{v}) \in \{\overline{0}, \overline{2}\}^{2}$ tel que $(\overline{u},\overline{a_{1}},\ldots,\overline{a_{2m-1}},\overline{a_{2m}-a_{2m+2}},\overline{v})$ est une solution de $(E_{4})$ (hypothèse de récurrence). Donc, $(\overline{v},\overline{u},\overline{a_{1}},\ldots,\overline{a_{2m}-a_{2m+2}})$ est une solution de $(E_{4})$ (invariance par permutations circulaires). De plus, $(\overline{a_{2m+2}},\overline{a_{2m+1}},\overline{a_{2m+2}},\overline{a_{2m+1}})$ est une solution de $(E_{4})$. Ainsi, le $(2m+2)$-uplet ci-dessous est une solution de $(E_{4})$:

\[(\overline{v},\overline{u},\overline{a_{1}},\ldots,\overline{a_{2m}-a_{2m+2}}) \oplus (\overline{a_{2m+2}},\overline{a_{2m+1}},\overline{a_{2m+2}},\overline{a_{2m+1}})=(\overline{v+a_{2m+1}},\overline{u},\overline{a_{1}},\ldots,\overline{a_{2m}},\overline{a_{2m+1}},\overline{a_{2m+2}}).\]

\noindent Donc, $(\overline{u},\overline{a_{1}},\ldots,\overline{a_{2m}},\overline{a_{2m+1}},\overline{a_{2m+2}},\overline{v+a_{2m+1}})$ est une solution de $(E_{4})$ et $\overline{u}, \overline{v+a_{2m+1}} \in \{\overline{0}, \overline{2}\}$. Par récurrence, le résultat est démontré.
\\
\\En utilisant la formule décrivant $M_{2m}(\overline{a_{1}},\ldots,\overline{a_{n}})$ avec des continuants, on déduit de ce qui précède que pour tout $2m$-uplet $(\overline{a_{1}},\ldots,\overline{a_{2m}})$ ne contenant que $\overline{0}$ et $\overline{2}$, $K_{2m}(\overline{a_{1}},\ldots,\overline{a_{2m}})=\pm \overline{1}$. 
\\
\\De plus, si $M_{n}(\overline{a_{1}},\ldots,\overline{a_{n}})=\overline{\epsilon} Id$ ($n \geq 4$), on a $\overline{a_{1}}=-\overline{\epsilon} K_{n-3}(\overline{a_{3}},\ldots,\overline{a_{n-1}})$ et $\overline{a_{n}}=-\overline{\epsilon} K_{n-3}(\overline{a_{2}},\ldots,\overline{a_{n-2}})$. Ainsi, pour chaque $2m$-uplet $(\overline{a_{1}},\ldots,\overline{a_{2m}})$ ne contenant que $\overline{0}$ et $\overline{2}$, il existe une unique solution de taille $2m+2$ de la forme $(\overline{x},\overline{a_{1}},\ldots,\overline{a_{2m}},\overline{y})$. Donc, si $\mathcal{E}_{2m+2}$ désigne l'ensemble des solutions de $(E_{4})$ de taille $2m+2$ ne contenant que $\overline{0}$ et $\overline{2}$, on a que
\[\begin{array}{ccccc} 
f_{2m+2} & : & \mathcal{E}_{2m+2} & \longrightarrow & \{\overline{0},\overline{2}\}^{2m} \\
 & & (\overline{x},\overline{a_{1}},\ldots,\overline{a_{2m}},\overline{y})  & \longmapsto & (\overline{a_{1}},\ldots,\overline{a_{2m}})
\end{array}\]

\noindent est une bijection. En particulier, il y a $2^{2m}$ solutions de $(E_{4})$ de taille $2m+2$ ne contenant que $\overline{0}$ et $\overline{2}$.
\\
\\ii) On dispose des solution suivantes de $(E_{4})$ : $(\overline{0},\overline{1},\overline{1},\overline{1},\overline{0})$, $(\overline{2},\overline{-1},\overline{1},\overline{-1},\overline{2})$, $(\overline{2},\overline{1},\overline{-1},\overline{1},\overline{2})$, $(\overline{0},\overline{1},\overline{-1},\overline{-1},\overline{2})$, $(\overline{0},\overline{-1},\overline{1},\overline{1},\overline{2})$, $(\overline{2},\overline{1},\overline{1},\overline{-1},\overline{0})$, $(\overline{2},\overline{-1},\overline{-1},\overline{1},\overline{0})$, $(\overline{0},\overline{-1},\overline{-1},\overline{-1},\overline{0})$. En utilisant ces dernières et le fait que $\pm \overline{1}+\overline{0}$ et $\pm \overline{1}+\overline{2}$ sont dans $\{-\overline{1}, \overline{1}\}$, la preuve du point ii) est similaire à celle de i).

\end{proof}
		
\subsubsection{Quelques éléments obtenus informatiquement} 	

Pour un anneau commutatif unitaire $A$, on note $\ell_{A}$ la taille maximale des $\lambda$-quiddités irréductibles sur $A$. Dans \cite{CM}, on avait obtenu, avec un programme informatique, un certain nombre de résultats sur $\ell_{A}$ dans les cas des anneaux $\mathbb{Z}/N\mathbb{Z}$ (voir \cite{CM} section 5.2 et Annexe B). En combinant ces derniers avec le corollaire \ref{33} et le lemme chinois, on a : 

\medskip

\begin{center}
\begin{tabular}{|c|c|c|c|c|}
\hline
$A$     & $\mathbb{Z}/2\mathbb{Z} \times \mathbb{Z}/5\mathbb{Z}$ & $\mathbb{Z}/3\mathbb{Z} \times \mathbb{Z}/4\mathbb{Z}$ & $\mathbb{Z}/2\mathbb{Z} \times \mathbb{Z}/7\mathbb{Z}$ & $\mathbb{Z}/3\mathbb{Z} \times \mathbb{Z}/5\mathbb{Z}$  \rule[-7pt]{0pt}{18pt} \\
\hline
$\ell_{A}$ & 12 & 15 & 20 & 26 
\rule[-7pt]{0pt}{18pt} \\
  \hline
\end{tabular}
\end{center}

\noindent Avec un autre programme informatique, on a obtenu que si $A=\mathbb{Z}/3\mathbb{Z} \times \mathbb{Z}/3\mathbb{Z}$ alors $\ell_{A}=12$.

\medskip
					
\subsection{$\lambda$-quiddités sur des produits direct et anneaux de caractéristique $0$}

L'objectif de cette section est de démontrer le théorème \ref{28}. Pour cela, on va donner plusieurs résultats intermédiaires. On commence par le lemme suivant :

\begin{lemma}
\label{45}

Soient $A$ un anneau commutatif unitaire et $n \geq 4$. $M_{n}(1_{A},n_{A}-2_{A},1_{A},2_{A},\ldots,2_{A})=-Id$.

\end{lemma}

On va donner deux preuves de ce résultat, une géométrique et une algébrique. Pour la preuve géométrique, on va utiliser le résultat suivant qui est la généralisation à un anneau quelconque du théorème de Conway-Coxeter (voir par exemple \cite{M} Théorème 2 pour une formulation matricielle de ce dernier).

\begin{proposition}
\label{46}

On considère la triangulation d'un polygone convexe $P$ à $n$ sommets par des diagonales ne se coupant qu'aux sommets. À chaque sommet de $P$ on associe un élément $c=k_{A}$ où $k$ est le nombre de triangles utilisant le sommet considéré. On parcourt les sommets, à partir de n'importe lequel d'entre eux, dans le sens horaire ou le sens trigonométrique, pour obtenir le $n$-uplet $(c_{1},\ldots,c_{n})$. Ce $n$-uplet est la quiddité de la triangulation de $P$. On a :
\[M_{n}(c_{1},\ldots,c_{n})=-Id.\]

\end{proposition}

\noindent Afin d'avoir une présentation complète, on va en redonner la preuve.

\begin{proof}

On procède par récurrence sur $n$. Si $n=3$, la quiddité de la triangulation de $P$ est $(c_{1},\ldots,c_{n})=(1_{A},1_{A},1_{A})$ et $M_{n}(c_{1},\ldots,c_{n})=-Id$. On suppose qu'il existe un $n \geq 3$ tel que toutes les quiddités des triangulations des polygones convexes à $n$ sommets vérifient la condition souhaitée. Soit $P$ un polygone convexe à $n+1$ sommets. On considère une triangulation de $P$ par des diagonales ne se coupant qu'aux sommets. On choisit un sommet que l'on numérote 1 puis on numérote les autres en suivant le sens trigonométrique. On note $(c_{1},\ldots,c_{n+1})$ la quiddité de cette triangulation (obtenue avec la numérotation précédente). Cette dernière possède un triangle extérieur, c'est-à-dire un triangle dont deux côtés sont des côtés de $P$. Ainsi, il existe $i$ dans $[\![1;n+1]\!]$ tel que $c_{i}=1_{A}$. En supprimant ce triangle, on obtient une triangulation dont la quiddité est $(c_{1},\ldots,c_{i-1}-1_{A},c_{i+1}-1_{A},\ldots,c_{n})$. Par récurrence, on a l'égalité : $M_{n}(c_{1},\ldots,c_{i-1}-1_{A},c_{i+1}-1_{A},\ldots,c_{n})=-Id$. De plus, on a les relations suivantes : $(c_{1},\ldots,c_{n}) \sim (c_{i+1},\ldots,c_{n},c_{1},\ldots,c_{i})=(c_{i+1}-1_{A},\ldots,c_{n},c_{1},\ldots,c_{i-1}-1_{A}) \oplus (1_{A},1_{A},1_{A})$ et $M_{3}(1_{A},1_{A},1_{A})=-Id$. Donc, on a $M_{n+1}(c_{1},\ldots,c_{n+1})=-Id$. Par récurrence, le résultat est démontré.

\end{proof}

\begin{proof}[Preuve géométrique du lemme \ref{45}]

On considère un polygone convexe à $n$ sommets. On choisit un sommet que l'on numérote 1 et on numérote les autres sommets en suivant le sens trigonométrique. Pour tout $i \in [\![4;n]\!]$ on relie le sommet $i$ au sommet 2. Comme une triangulation d'un polygone convexe à $n$ sommets contient $n-2$ triangles, on obtient une triangulation de quiddité :
\[(c_{1},\ldots,c_{n})=(1_{A},n_{A}-2_{A},1_{A},2_{A},\ldots,2_{A}).\]
\noindent Par la proposition \ref{46}, le résultat est démontré. Si on note $p_{n}=n_{A}-2_{A}$, cela donne :
$$
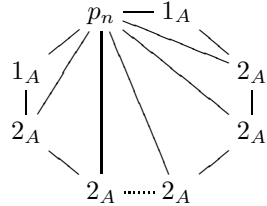
\shorthandoff{; :!?}
\xymatrix @!0 @R=0.40cm @C=0.5cm
{
&&p_{n}\ar@{-}[lldd]\ar@{-}[dddddd]\ar@{-}[rrrrdddd] \ar@{-}[rr]&&1_{A}\ar@{-}[rrdd]&
\\
&&&
\\
1_{A}\ar@{-}[dd]&&&&&& 2_{A} \ar@{-}[dd]\ar@{-}[lllluu]
\\
&&&&
\\
2_{A}\ar@{-}[rruuuu]&&&&&& 2_{A}
\\
&&&
\\
&&2_{A}\ar@{.}[rr] \ar@{-}[lluu] &&2_{A} \ar@{-}[rruu]\ar@{-}[lluuuuuu]
}
$$
\end{proof}

On va maintenant donner une preuve algébrique dont l'ingrédient principal est le lemme classique ci-dessous (voir par exemple \cite{M2} lemme 3.19): 

\begin{lemma}
\label{47}

Soit $n \in \mathbb{N}$. $K_{n}(2_{A},\ldots,2_{A})=n_{A}+1_{A}$ et, si $n \geq 1$, $M_{n}(2_{A},\ldots,2_{A})=\begin{pmatrix}
   n_{A}+1_{A}   & -n_{A} \\
   n_{A} & -n_{A}+1_{A}
\end{pmatrix}$.

\end{lemma}

\begin{proof}

On raisonne par récurrence double sur $n$. Si $n=0$ ou si $n=1$ alors le résultat est vrai. On suppose qu'il existe $n \in \mathbb{N}^{*}$ tel que $K_{n}(2_{A},\ldots,2_{A})=n_{A}+1_{A}$ et $K_{n-1}(2_{A},\ldots,2_{A})=n_{A}$. En développant le déterminant définissant $K_{n+1}(2_{A},\ldots,2_{A})$ suivant la première colonne, on a :
\[K_{n+1}(2_{A},\ldots,2_{A})=2_{A}K_{n}(2_{A},\ldots,2_{A})-K_{n-1}(2_{A},\ldots,2_{A})=2_{A}n_{A}+2_{A}-n_{A}=(n_{A}+1_{A})+1_{A}.\]

\noindent La formule est vraie pour $n+1$ et donc par récurrence elle est vraie pour tout $n$. On note qu'elle est également vraie pour $n=-1$. En utilisant la formule de $M_{n}(a_{1},\ldots,a_{n})$ qui exprime la matrice avec des continuants, on a la formule matricielle souhaitée.

\end{proof}

\begin{proof}[Preuve algébrique du lemme \ref{45}]

Soit $n \geq 4$ et $M=M_{n}(1_{A},n_{A}-2_{A},1_{A},2_{A},\ldots,2_{A})$. On a 
\begin{eqnarray*}
M &=& M_{n-3}(2_{A},\ldots,2_{A})M_{1}(1_{A})M_{1}(n_{A}-2_{A})M_{1}(1_{A}) \\
  &=& \begin{pmatrix}
   n_{A}-2_{A}   & -n_{A}+3_{A} \\
   n_{A}-3_{A} & -n_{A}+4_{A}
\end{pmatrix}\begin{pmatrix}
   1_{A}   & -1_{A} \\
   1_{A} & 0_{A}
\end{pmatrix}\begin{pmatrix}
   n_{A}-2_{A}   & -1_{A} \\
   1_{A} & 0_{A}
\end{pmatrix}\begin{pmatrix}
   1_{A}   & -1_{A} \\
   1_{A} & 0_{A}
\end{pmatrix} \\
   &=& \begin{pmatrix}
   1_{A}   & -n_{A}+2_{A} \\
   1_{A} & -n_{A}+3_{A}
\end{pmatrix} \begin{pmatrix}
   n_{A}-3_{A}   & -n_{A}+2_{A} \\
   1_{A} & -1_{A}
\end{pmatrix} \\
   &=& \begin{pmatrix}
   -1_{A}   & 0_{A} \\
   0_{A} & -1_{A}
\end{pmatrix}.
\end{eqnarray*}

\end{proof}

La prochaine étape pour effectuer la démonstration du théorème \ref{28} est de s'intéresser à la réduction sur un anneau $A$ de caractéristique $0$ de la solution donnée par le lemme \ref{45}.

\begin{lemma}
\label{48}

Soit $A$ un anneau commutatif unitaire de caractéristique $0$. Soient $n \geq 4$, $k$ appartenant à $[\![3;n-1]\!]$ et $(b_{1},\ldots,b_{k}) \in A^{k}$. Si on peut utiliser $(b_{1},\ldots,b_{k})$ pour réduire $(1_{A},n_{A}-2_{A},1_{A},2_{A},\ldots,2_{A})$ alors $(b_{1},\ldots,b_{k})$ est de la forme $(x,1_{A},2_{A},\ldots,2_{A},y)$ ou $(x,2_{A},\ldots,2_{A},1_{A},y)$ (avec éventuellement un nombre nul de $2_{A}$).

\end{lemma}

\begin{proof}

Notons $(a_{1},\ldots,a_{n})=(1_{A},n_{A}-2_{A},1_{A},2_{A},\ldots,2_{A})$. Si $k=3$ alors nécessairement $(b_{1},\ldots,b_{k})=(1_{A},1_{A},1_{A})$ puisque $(a_{1},\ldots,a_{n})$ ne contient pas $-1_{A}$. On suppose maintenant $n \geq 5$ et $k \geq 4$. Notons que $(a_{n},\ldots,a_{1})$ peut-être obtenu par permutations circulaires de $(a_{1},\ldots,a_{n})$. Donc, si on peut utiliser $(b_{1},\ldots,b_{k})$ pour réduire $(a_{1},\ldots,a_{n})$, il existe en entier $h$ et une $\lambda$-quiddité sur $A$ $(c_{1},\ldots,c_{l})$ tels que $(a_{h+1},\ldots,a_{n},a_{1},\ldots,a_{h})=(c_{1},\ldots,c_{l}) \oplus (b_{1},\ldots,b_{k})$.
\\
\\Dans la suite, on va utiliser la fait suivant qui découle des propriétés de $\oplus$ : si $(c_{1},\ldots,c_{i-1},1_{A},c_{i+1},\ldots,c_{n})$ est une $\lambda$-quiddité sur $A$ alors $(c_{1},\ldots,c_{i-2},c_{i-1}-1_{A},c_{i+1}-1_{A},c_{i+2},\ldots,c_{n})$ l'est aussi.
\\
\\Supposons que l'on peut utiliser $(b_{1},\ldots,b_{k})$ pour réduire $(a_{1},\ldots,a_{n})$. Il existe un entier $j$ non nul tel que $(b_{2},\ldots,b_{k-1})=(a_{j},\ldots,a_{j+k-3})$ et $(b_{1},\ldots,b_{k})$ est une solution de \eqref{a}. On va considérer plusieurs cas :
\\
\\i) Si $(b_{2},\ldots,b_{k-1})$ ne contient pas de $1_{A}$. Nécessairement, $(b_{1},\ldots,b_{k})=(x,2_{A},\ldots,2_{A},y)$. Ceci est impossible car, par le lemme \ref{47}, $K_{k-2}(2_{A},\ldots,2_{A})=k_{A}-1_{A} \neq \pm 1_{A}$ (car $A$ est de caractéristique 0 et et $k \geq 4$).
\\
\\ii) Si $(b_{2},\ldots,b_{k-1})$ contient deux $1_{A}$. On a $k \geq 5$ et on distingue deux cas :
\begin{itemize}
\item $(b_{1},\ldots,b_{k})=(x,\underbrace{2_{A},\ldots,2_{A}}_{l},1_{A},n_{A}-2_{A},1_{A},\underbrace{2_{A},\ldots,2_{A}}_{k-5-l},y)$ avec $0 \leq l \leq k-5$. Si $l \geq 1$, on a, en utilisant le $1_{A}$ à gauche, $(x,\underbrace{2_{A},\ldots,2_{A}}_{l-1},1_{A},n_{A}-3_{A},1_{A},\underbrace{2_{A},\ldots,2_{A}}_{k-5-l},y)$ solution de \eqref{a}. En réitérant ce processus autant que possible, à droite et à gauche, on arrive à une solution de la forme $(u,n_{A}-k_{A}+1_{A},v)$, ce qui est impossible car $n_{A}-k_{A}+1_{A} \neq \pm 1_{A}$ (puisque $A$ est de caractéristique 0 et $4 \leq k \leq n-1$).
\item $(b_{1},\ldots,b_{k})=(x,1_{A},2_{A},\ldots,2_{A},1_{A},y)$ et $k=n+1$, ce qui est impossible.
\\
\end{itemize}

\noindent iii) Si $(b_{2},\ldots,b_{k-1})$ contient un seul $1_{A}$. On distingue plusieurs cas :
\begin{itemize}
\item $(b_{1},\ldots,b_{k})=(x,2_{A},\ldots,2_{A},1_{A},n_{A}-2_{A},y)$. En utilisant le $1_{A}$ et en procédant comme au-dessus, on a $(x-1_{A},n_{A}-k_{A}+1_{A},y)$ solution de \eqref{a}, ce qui est impossible.
\item $(b_{1},\ldots,b_{k}) \sim (x,1_{A},n_{A}-2_{A},y)$. Ce cas est impossible car $1_{A} \times (n_{A}-2_{A}) \neq 0_{A}, 2_{A}$ (puisque $A$ est de caractéristique 0 et $n \geq 5$)
\item $(b_{1},\ldots,b_{k})=(x,n_{A}-2_{A},1_{A},2_{A},\ldots,2_{A},y)$. En utilisant le $1_{A}$ et en procédant comme au-dessus, on a $(x,n_{A}-k_{A}+1_{A},y-1_{A})$ solution de \eqref{a}, ce qui est impossible.
\item $(b_{1},\ldots,b_{k})=(x,1_{A},2_{A},\ldots,2_{A},y)$ ou $(b_{1},\ldots,b_{k})=(x,2_{A},\ldots,2_{A},1_{A},y)$.
\\
\end{itemize}

\noindent Ainsi, $(b_{1},\ldots,b_{k})=(x,1_{A},2_{A},\ldots,2_{A},y)$ ou $(b_{1},\ldots,b_{k})=(x,2_{A},\ldots,2_{A},1_{A},y)$.

\end{proof}

\begin{lemma}
\label{49}

Soit $A$ un anneau commutatif unitaire de caractéristique $0$. Soient $n \geq 4$, $k$ appartenant à $[\![3;n-1]\!]$ et $(b_{1},\ldots,b_{k}) \in A^{k}$. $(b_{1},\ldots,b_{k})$ peut réduire la solution $(1_{A},n_{A}-2_{A},1_{A},2_{A},\ldots,2_{A})$ si et seulement si  $(b_{1},\ldots,b_{k})=(k_{A}-2_{A},1_{A},2_{A},\ldots,2_{A},1_{A})$ ou $(b_{1},\ldots,b_{k})=(1_{A},2_{A},\ldots,2_{A},1_{A},k_{A}-2_{A})$ (avec éventuellement un nombre nul de $2_{A}$).

\end{lemma}

\begin{proof}

Soient $n \geq 4$ et $k \in [\![3;n-1]\!]$. 
\\
\\Par le lemme \ref{45} et l'invariance par permutations circulaires des solutions, $(k_{A}-2_{A},1_{A},2_{A},\ldots,2_{A},1_{A})$ et $(1_{A},2_{A},\ldots,2_{A},1_{A},k_{A}-2_{A})$ sont des solutions de \eqref{a}. De plus, on a :
\[(1_{A},n_{A}-2_{A},1_{A},2_{A},\ldots,2_{A}) \sim (1_{A},2_{A},\ldots,2_{A},1_{A},k_{A}-2_{A}) \oplus ((n_{A}+2_{A}-k_{A})-2_{A},1_{A},2_{A},\ldots,2_{A},1_{A}).\]

\noindent Si $(b_{1},\ldots,b_{k})$ peut réduire la solution $(1_{A},n_{A}-2_{A},1_{A},2_{A},\ldots,2_{A})$. Si $k=3$, $(b_{1},\ldots,b_{k})=(1_{A},1_{A},1_{A})$. On suppose maintenant $n \geq 5$ et $k \geq 4$. Par le lemme \ref{48}, $(b_{1},\ldots,b_{k})$ est de la forme $(x,1_{A},2_{A},\ldots,2_{A},y)$ ou $(x,2_{A},\ldots,2_{A},1_{A},y)$. Si $(b_{1},\ldots,b_{k})$ est une solution de la forme $(x,1_{A},2_{A},\ldots,2_{A},y)$. Par le lemme \ref{47}, on a $K_{k-2}(1_{A},2_{A},\ldots,2_{A})=K_{k-3}(2_{A},\ldots,2_{A})-K_{k-4}(2_{A},\ldots,2_{A})=k_{A}-2_{A}-(k_{A}-3_{A})=1_{A}$. Ainsi, $M_{k}(b_{1},\ldots,b_{k})=-Id$. De plus, on a 
\begin{itemize}
\item $x=K_{k-3}(2_{A},\ldots,2_{A})=k_{A}-2_{A}$;
\item $y=K_{k-3}(1_{A},2_{A},\ldots,2_{A})=K_{k-4}(2_{A},\ldots,2_{A})-K_{k-5}(2_{A},\ldots,2_{A})=(k_{A}-3_{A})-(k_{A}-4_{A})=1_{A}$.
\end{itemize}
\noindent Si $(b_{1},\ldots,b_{k})$ est une solution de la forme $(x,2_{A},\ldots,2_{A},1_{A},y)$, on procède de la même manière.

\end{proof}

\begin{remark}
{\rm 
Les deux lemmes ci-dessus ne sont plus vrais si l'anneau est de caractéristique non nulle. Par exemple, si $A=\mathbb{Z}/N\mathbb{Z}$, avec $N \geq 3$, et si $n \geq N+1$ on peut utiliser la solution de taille $N$ $(\overline{2},\ldots,\overline{2})$ (voir \cite{M2} Théorème 2.6) pour réduire la solution du lemme \ref{45}.
}
\end{remark}

\noindent On peut maintenant démontrer le résultat souhaité :

\begin{proof}[Démonstration du théorème \ref{28}]

i) On commence par le cas de deux anneaux.
\\
\\Soient $A$ et $B$ deux anneaux commutatifs unitaires de caractéristique $0$. Soit $n \geq 3$. Si $n=3$, il existe une $\lambda$-quiddité irréductible de taille 3 sur $A \times B$ (lemme \ref{35}). On suppose $n \geq 4$.
\\
\\On pose $(a_{1},\ldots,a_{n})=(1_{A},n_{A}-2_{A},1_{A},2_{A},\ldots,2_{A})$ et $(b_{1},\ldots,b_{n})=(2_{B},1_{B},n_{B}-2_{B},1_{B},2_{B},\ldots,2_{B})$. Par le théorème \ref{25} et le lemme \ref{45}, $((a_{1},b_{1}),\ldots,(a_{n},b_{n}))$ est une $\lambda$-quiddité sur $A \times B$. 
\\
\\Supposons par l'absurde que celle-ci est réductible. Par le théorème \ref{25}, il existe $\sigma \in D_{n}$, $3 \leq l, k \leq n-1$, $(c_{1},\ldots,c_{l})$, $(d_{1},\ldots,d_{k})$ deux $\lambda$-quiddités sur $A$ et $(c_{1}',\ldots,c_{l}')$, $(d_{1}',\ldots,d_{k}')$ deux $\lambda$-quiddités sur $B$ tels que :
\[\left\{
    \begin{array}{ll}
        (a_{\sigma.1},\ldots,a_{\sigma.n})=(c_{1},\ldots,c_{l}) \oplus (d_{1},\ldots,d_{k}); \\
        (b_{\sigma.1},\ldots,b_{\sigma.n})=(c_{1}',\ldots,c_{l}') \oplus (d_{1}',\ldots,d_{k}').
    \end{array}
\right. \\ \]

\noindent Par le lemme \ref{49}, $(d_{1},\ldots,d_{k})=(k_{A}-2_{A},1_{A},2_{A},\ldots,2_{A},1_{A})$ ou $(d_{1},\ldots,d_{k})=(1_{A},2_{A},\ldots,2_{A},1_{A},k_{A}-2_{A})$ et $(d_{1}',\ldots,d_{k}')=(k_{B}-2_{B},1_{B},2_{B},\ldots,2_{B},1_{B})$ ou $(d_{1}',\ldots,d_{k}')=(1_{B},2_{B},\ldots,2_{B},1_{B},k_{B}-2_{B})$. On va distinguer les cas :
\begin{itemize}
\item Si $(d_{1},\ldots,d_{k})=(k_{A}-2_{A},1_{A},2_{A},\ldots,2_{A},1_{A})$ et $(d_{1}',\ldots,d_{k}')=(k_{B}-2_{B},1_{B},2_{B},\ldots,2_{B},1_{B})$. $(a_{\sigma.1},\ldots,a_{\sigma.n})$ et $(b_{\sigma.1},\ldots,b_{\sigma.n})$ ont un 1 à la position $l+1$. Or, $(a_{\sigma.1},\ldots,a_{\sigma.n})$ n'a que deux $1_{A}$, en position $\sigma^{-1}.1$ et $\sigma^{-1}.3$, et $(b_{\sigma.1},\ldots,b_{\sigma.n})$ n'a que deux $1_{B}$, en position $\sigma^{-1}.2$ et $\sigma^{-1}.4$. Ceci est absurde.
\item Si $(d_{1},\ldots,d_{k})=(1_{A},2_{A},\ldots,2_{A},1_{A},k_{A}-2_{A})$ et $(d_{1}',\ldots,d_{k}')=(1_{B},2_{B},\ldots,2_{B},1_{B},k_{B}-2_{B})$. En procédant comme dans le cas précédent, on arrive à une absurdité.
\item Si $(d_{1},\ldots,d_{k})=(k_{A}-2_{A},1_{A},2_{A},\ldots,2_{A},1_{A})$ et $(d_{1}',\ldots,d_{k}')=(1_{B},2_{B},\ldots,2_{B},1_{B},k_{B}-2_{B})$. $(b_{\sigma.1},\ldots,b_{\sigma.n})$ doit avoir $n_{B}-2_{B}$ en position $1$, ce qui implique $\sigma^{-1}.3=1$. On a donc $\sigma=r^{2}$ ou $\sigma=sr^{n-3}$. De plus, $(a_{\sigma.1},\ldots,a_{\sigma.n})$ doit avoir $n_{A}-2_{A}$ en position $l$, ce qui implique $\sigma^{-1}.2=l$. Ceci donne $l=2$ ou $l=n$, ce qui est absurde.
\item Si $(d_{1},\ldots,d_{k})=(1_{A},2_{A},\ldots,2_{A},1_{A},k_{A}-2_{A})$ et $(d_{1}',\ldots,d_{k}')=(k_{B}-2_{B},1_{B},2_{B},\ldots,2_{B},1_{B})$. En procédant comme dans le cas précédent, on arrive à une absurdité.
\end{itemize}

\noindent ii) On considère maintenant le cas général. 
\\
\\Soient $I$ un ensemble contenant au moins deux éléments et $(A_{i})_{i \in I}$ une famille d'anneaux commutatifs unitaires indexées par $I$. Par hypothèse, il existe $i_{0}$ et $i_{1}$ dans $I$ tels que $A_{i_{0}}$ et $A_{i_{1}}$ sont de caractéristique $0$. Soit $n \geq 3$. Si $n=3$, il existe une $\lambda$-quiddité irréductible de taille 3 sur $\prod_{i \in I} A_{i}$ (lemme \ref{35}). On suppose $n \geq 4$.
\\
\\Soit $B=A_{i_{0}}$. On pose $(a_{i_{0},1},\ldots,a_{i_{0},n})=(1_{B},n_{B}-2_{B},1_{B},2_{B},\ldots,2_{B})$. Soit $i$ dans $I$ différent de $i_{0}$, on pose $(a_{i,1},\ldots,a_{i,n})=(2_{A_{i}},1_{A_{i}},n_{A_{i}}-2_{A_{i}},1_{A_{i}},2_{A_{i}},\ldots,2_{A_{i}}) $. Par le théorème \ref{25} et le lemme \ref{45}, $((a_{i,1})_{i},\ldots,(a_{i,n})_{i})$ est une $\lambda$-quiddité sur $\prod_{i \in I} A_{i}$. 
\\
\\Supposons par l'absurde que celle-ci est réductible. Par le théorème \ref{25}, $(a_{i_{0},1},\ldots,a_{i_{0},n})$ et $(a_{i_{1},1},\ldots,a_{i_{1},n})$ sont simultanément réductibles. Par le cas i), ceci est absurde. Donc, $((a_{i,1})_{i},\ldots,(a_{i,n})_{i})$ est une $\lambda$-quiddité irréductible de taille $n$ sur $\prod_{i \in I} A_{i}$. 

\end{proof}

Comme $\mathbb{Z}$ est de caractéristique $0$, le corollaire \ref{29} découle immédiatement du théorème \ref{28}. Notons que le cas de $\mathbb{Z} \times \mathbb{Z}$ est très différent de celui de $\mathbb{Z}$ (voir \cite{C} Théorème 3.2).

\begin{examples}
{\rm Soit $n \geq 3$. Il existe une $\lambda$-quiddité irréductible de taille $n$ sur :
\begin{itemize}
\item  $\mathbb{Z}[X] \times \mathbb{Z}[X]$;
\item  $\mathbb{Z} \times \mathbb{Z}[\sqrt{2}]$;
\item $\mathbb{Z} \times (\mathbb{Z}/2\mathbb{Z}) \times \mathbb{Z}$;
\item  $\mathbb{Z}[\sqrt{2}] \times \mathbb{Z}[\sqrt{3}]$.
\end{itemize}
\noindent Notons que, dans le dernier cas, on n'a pas d'information précise sur les $\lambda$-quiddités irréductibles sur $\mathbb{Z}[\sqrt{2}]$ et sur $\mathbb{Z}[\sqrt{3}]$.
}
\end{examples}

\end{document}